\documentclass[11pt,reqno]{amsart}
\usepackage[norelsize]{algorithm2e}
\usepackage{amsmath}
\usepackage{amsthm}
\usepackage{amssymb}
\usepackage{mathabx}
\usepackage{graphicx}
\usepackage{tikz}
\usetikzlibrary{matrix,arrows}
\usepackage[margin=1.2in]{geometry}
\usepackage{graphicx}
\usepackage{enumerate}
\usepackage{cite}
\usepackage{setspace}
\usepackage{amsfonts}
\usepackage{mathtools}
\usepackage[colorlinks, citecolor=magenta, linkcolor=black]{hyperref}

\newtheoremstyle{slplain}% name
  {.8\baselineskip\@plus.2\baselineskip\@minus.2\baselineskip}% Space above
  {.4\baselineskip\@plus.2\baselineskip\@minus.2\baselineskip}% Space below
  {\slshape}% Body font
  {}%Indent amount (empty = no indent, \parindent = para indent)
  {\bfseries}%  Thm head font
  {.}%       Punctuation after thm head
  { }%      Space after thm head: " " = normal interword space;
  {}%       Thm head spec

\theoremstyle{slplain}
\newtheorem{theorem}{Theorem}
\newtheorem{lemma}{Lemma}

\newtheorem{corollary}{Corollary}

\newtheorem*{theorem*}{Theorem}
\newtheorem*{corollary*}{Corollary}
\newtheorem*{remark}{Remark}

\newcommand\numberthis{\addtocounter{equation}{1}\tag{\theequation}}

\begin{document}

\title[Slice Implies Mutant-Ribbon]{Slice Implies Mutant-Ribbon \protect\\ for Odd, 5-Stranded Pretzel Knots}
\author{Kathryn A. Bryant}

\begin{abstract}
A pretzel knot $K$ is called {\it odd} if all its twist parameters are odd, and {\it mutant ribbon} if it is mutant to a simple ribbon knot. We prove that the family of odd, 5-stranded pretzel knots satisfies a weaker version of the Slice-Ribbon Conjecture: All slice, odd, 5-stranded pretzel knots are {\it mutant ribbon}. We do this in stages by first showing that 5-stranded pretzel knots having twist parameters with all the same sign or with exactly one parameter of a different sign have infinite order in the topological knot concordance group, and thus in the smooth knot concordance group as well. Next, we show that any odd, 5-stranded pretzel knot with zero pairs or with exactly one pair of canceling twist parameters is not slice. 
\end{abstract}

\maketitle

\section{Introduction} \label{intro}
A knot $K \subset S^3$ is {\it smoothly slice} if it bounds a smoothly embedded disk in the $4$-ball. The notion of smoothly slice knots can be used to define the smooth knot concordance group $\mathcal{C}$ under the operation of connected sum. It is a widely-studied group for which a slice knot represents the identity element. For explicit information about the concordance relation, see \cite{livknots}. Fine details of the group structure of $\mathcal{C}$ continue to elude mathematicians, but concordance order is one small way of gaining insights into $\mathcal{C}$. The topic of determining slice knots and concordance order for knots within families of pretzel knots has also been studied with increasing frequency over the past 30 years and various results can be found in \cite{greene-jabuka}, \cite{lecuona}, \cite{miller}, \cite{HKL}, and \cite{long}. 

The Slice-Ribbon Conjecture hypothesizes that a if a knot is slice, then it is also ribbon. Given that ribbon knots are easily seen to be slice, this is ultimately a conjecture about the equivalence of the notions `slice' and `ribbon'. Previous work by Joshua Greene and Stanislav Jabuka in \cite{greene-jabuka} on the Slice-Ribbon Conjecture for odd, $3$-stranded pretzel knots and work by Ana Lecuona in \cite{lecuona} on even pretzel knots inspired this project. This paper studies sliceness and concordance order for odd, 5-stranded pretzel knots. 

A {\it $k$-stranded pretzel link}, denoted $P(p_1, p_2, .\ .\ .\ ,p_k)$ where the $p_i \in \mathbb{Z} - \{0\}$ are called the {\it twist parameters}, is a knot in two cases: when exactly one of the twist parameters is even, or when $k$ is odd and all the twist parameters are odd. A pretzel knot is called {\it even} in the former case and {\it odd} in the latter. A {\it 0-pair pretzel knot} is a pretzel knot for which there are no canceling pairs of twist parameters satisfying $p_i = -p_j$. A {\it 1-pair pretzel knot} is a pretzel knot for which there exists a canceling pair of twist parameters, but when the pair is removed from the $k$-tuple defining the knot, the resulting $(k-2)$--stranded knot is 0-pair. Generally, a {\it t-pair pretzel knot} is one for which removing a single canceling pair of twist parameters results in a $(t-1)$-pair pretzel knot with two fewer strands. With this defintion, 5-stranded pretzel knots $P(a,b,c,d,e)$ can be $0$-pair, 1-pair, or 2-pair. See Figure \ref{pknot1}.

When proving statements about pretzel knots, it is often necessary to differentiate between the knots that contain twist parameters equal to $\pm 1$ and those that do not. If for $K = P(p_1, ..., p_k)$ there exists $i \in \{1, ..., k\}$ such that $p_i = \pm 1$, then we say $K$ is a {\it pretzel knot with single-twists}; otherwise, we say $K$ is a {\it pretzel knot without single-twists}. 
 
\begin{figure}
\includegraphics[height=175pt]{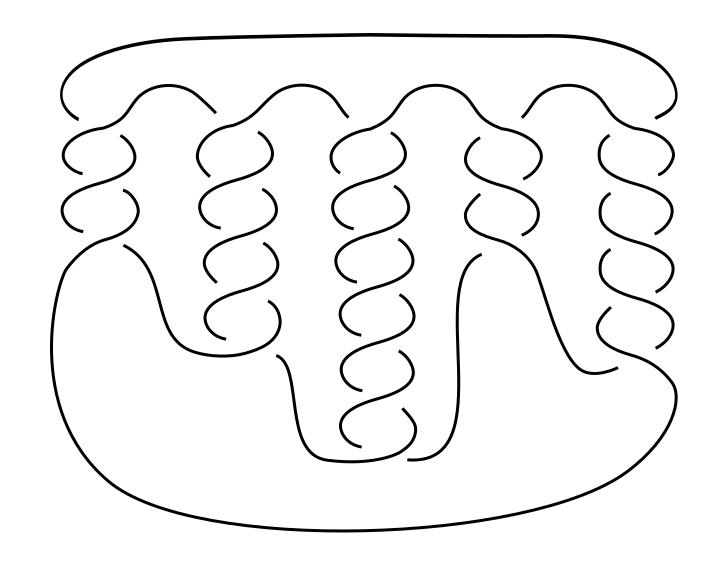}
\caption{Pretzel Knot $P(3,5,7,-3,-5)$}
\label{pknot1}
\end{figure}

The classification of pretzel knots appears in \cite{zieschang}, which classifies the much larger class of Montesinos knots of which pretzel knots are a special case. The classification gives that two pretzel knots without single-twists are smoothly isotopic if their twist parameters differ by cyclic permutations, reflections, or compositions thereof. Two pretzel knots {\it with} single-twists are smoothly isotopic if their twist parameters differ by cyclic permutations, reflections, and/or transpositions involving $\pm 1$-twisted strands. Two $k$-stranded pretzel knots whose twist parameters are equal as {\it unordered} $k$-tuples but not equal as {\it ordered} $k$-tuples are called {\it pretzel knot mutants}. This specific kind of mutation is the only type considered here, so ``mutation" from this point on will always mean ``pretzel knot mutation." 

Mutation is a crucial topic for the problem of determining sliceness for $k$-stranded pretzel knots when $k \geq 4$ because many knot invariants used to obstruct sliceness are unable to detect pretzel knot mutants. In fact, any knot invariant based on the double branched cover of $S^3$ along the knot will fail to detect pretzel knot mutants; Bedient shows in \cite{bedient} that any two pretzel knots defined by the same unordered $k$-tuple of twist parameters have the same double branched cover. Given a $k$-tuple  $(p_1, ... ,p_k)$ of twist parameters, $P\{p_1, ..., p_k\}$ will denote the set of pretzel knots having $\{p_1, ... , p_k\}$ as twist parameters, as well as all mirrors of such knots.  

Among pretzel knots is a subset of knots we will call {\it simple ribbon}. A {\it simple ribbon move} on a pretzel knot is the ribbon move shown in Figure \ref{ribbonmove}, performed always on the top-most twist of two adjacent strands of $K$ having canceling numbers of twists. We say a pretzel knot $K$ is {\it simple ribbon} if there exists a sequence of simple ribbon moves that reduces $K$ to a 1-stranded pretzel knot (if $K$ is odd) or to a 2-stranded pretzel knot $P(a,b)$ where $a = -b-1$ (if $K$ is even). A prerequisite for a pretzel knot to be simple ribbon is that if $K$ is $k$-stranded, then $K$ must be $(\frac{k-1}{2})$-pair. But, while all 1-pair, 3-stranded pretzel knots are simple ribbon, not all 2-pair, 5-stranded pretzel knots are simple ribbon. For example, the 2-pair knot $P(3, 5, -3, -5, 7)$ is not simple ribbon because no two adjacent strands have canceling numbers of twists. This phenomenon extends for all $k \geq 4$.

\begin{figure}
\includegraphics[height=400pt]{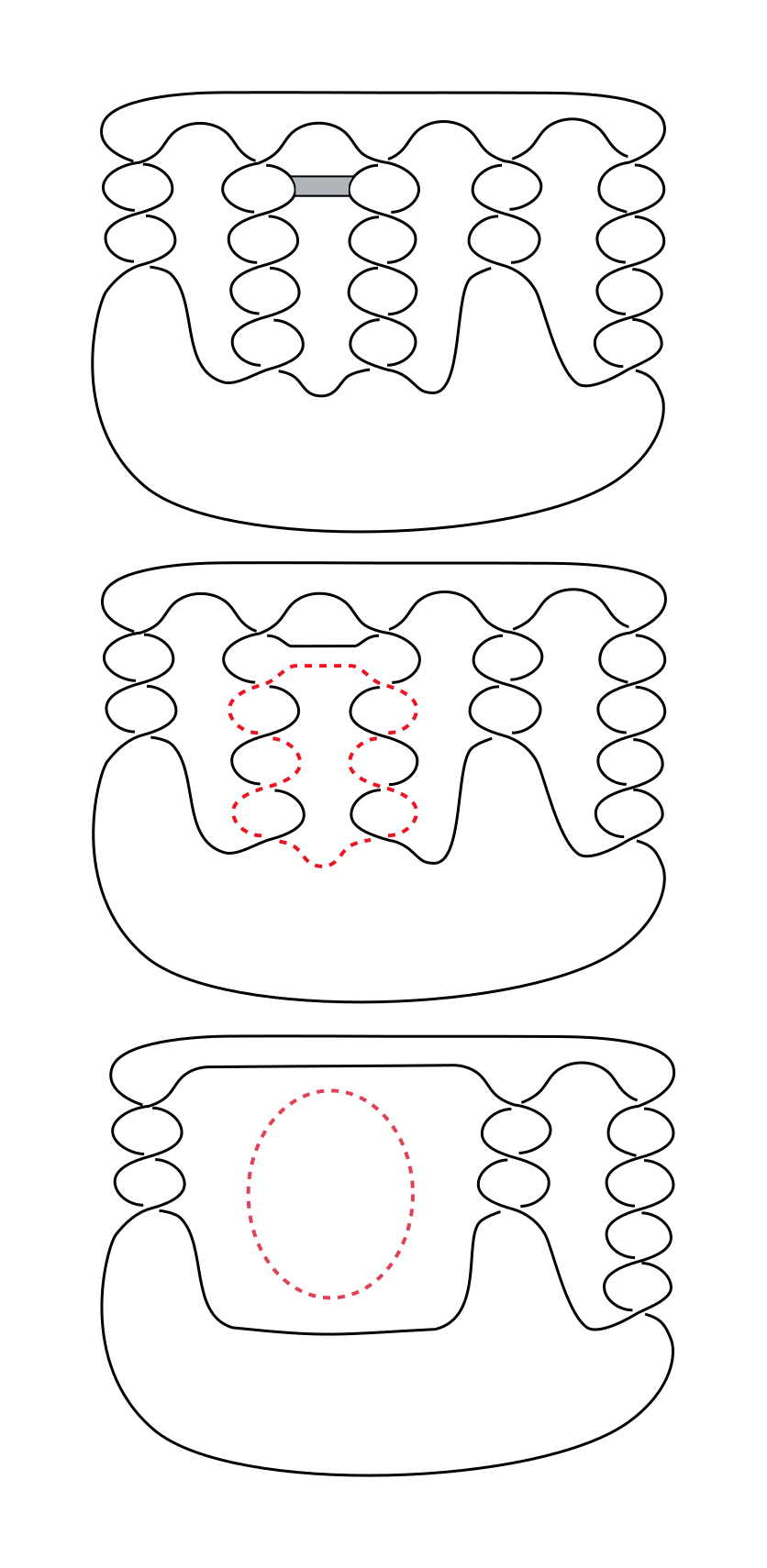}
\caption{Simple ribbon move on Pretzel Knot $P(-3,-5,5,3,-5)$}
\label{ribbonmove}
\end{figure}

\vspace{.05in}

\section{Results} \label{results}

As previously mentioned, this project was motivated by work of Joshua Greene and Stanislav Jabuka in \cite{greene-jabuka} on the Slice-Ribbon Conjecture for odd, $3$-stranded pretzel knots and by work of Ana Lecuona in \cite{lecuona} on even pretzel knots. In \cite{lecuona}, Lecuona writes down the following conjecture:
 
\noindent {\bf Pretzel Ribbon Conjecture.} (Lecuona) Let $K$ be a pretzel knot whose twist parameters are all greater than 1 in absolute value. If $K$ is ribbon, then $K$ is simple ribbon. \

For odd, 3-stranded pretzel knots, the Pretzel Ribbon Conjecture posits that the only ribbon knots are the simple ribbon knots, i.e. the 1-pairs. Similarly for odd, 5-stranded pretzel knots, it says that the only ribbon knots are the simple ribbon knots, which are 2-pairs {\it for which at least one of the canceling pairs is adjacent}. Greene and Jabuka show in \cite{greene-jabuka} that odd, 3-stranded pretzel knots satisfy both the Pretzel Ribbon Conjecture and the Slice-Ribbon Conjecture by proving that a knot of this type is slice if and only if it is 1-pair. This result, which proves the two aforementioned conjectures in a particular case, hints to the following possible strengthening of the Slice-Ribbon Conjecture in the specific case of pretzel knots:

\noindent {\bf Pretzel Slice-Ribbon Conjecture.} If $K$ is a slice pretzel knot, then $K$ is simple ribbon.

Of course, if the Pretzel Ribbon Conjecture is true then the above is equivalent to the original version of the Slice-Ribbon Conjecture. There is evidence that supports the Pretzel Slice-Ribbon Conjecture in the odd, 5-stranded case. In \cite{HKL}, the authors prove that $P(3,5,-3,-5,7)$ is {\it not} slice despite being mutant to the two simple ribbon knots $P(3,-3, 5, -5, 7)$ and $P(3, 5, -5, -3, 7)$. See Figure \ref{ribbonmove} for an illustration of the ribbon move on two adjacent strands in a 5-stranded pretzel knot whose twist numbers cancel.

This present work applies the techniques used by Jabuka and Greene on odd, 3-stranded pretzel knots to odd, 5-stranded pretzel knots, in the hope of showing that this new class of knots also satisfies the Pretzel Slice-Ribbon Conjecture as well. It should be noted that Greene and Jabuka went a step farther and proved that all non-slice, odd, 3-stranded pretzel knots have infinite order in $\mathcal{C}$. To obtain these results, they used three tools: the knot signature from classical knot theory, Donaldson's Diagonalization Theorem from gauge theory, and the $d$-invariant from Heegaard-Floer Theory. 

The main results of this project are given below, accompanied by brief explanations as to where each of the above three tools comes into play. In Theorem 1 and its corollary, $\sigma(K)$ denotes the signature of $K$; $s$ is the difference between the number of positive twist parameters and the number of negative twist parameters of $K$; $\hat{e}$ is the orbifold Euler characteristic of $K$ given by the sum of the reciprocals of the twist parameters; and $\text{sgn}()$ is the function returning -1, 0, or +1 according to whether the input is negative, zero, or positive, respectively. The first result is about about the larger class of odd, $k$-stranded pretzel knots where $k \geq 3$:

\begin{theorem}
\label{thm:sthm}
If $K$ is an odd, $k$-stranded pretzel knot, then $\sigma(K) = s - \text{sgn}(\hat{e})$. In particular, $\sigma(K) = 0$ if and only if $s = \text{sgn}(\hat{e})$. 
\end{theorem}

\vspace{.05in}

\begin{corollary}
\label{cor1}
All odd, $k$-stranded pretzel knots with $s \neq \text{sgn}(\hat{e})$ have infinite order in the topological knot concordance group $\mathcal{T}$.
\end{corollary}

The corollary follows from the fact that $\sigma$ is a homomorphism from $\mathcal{T} \rightarrow \mathbb{Z}$, and it implies infinite order in the smooth knot concordance group $\mathcal{C}$ as well. It is a well-known fact that we call on later that if a knot $K$ is slice, then $\sigma(K) = 0$. An implicit implication of Theorem 1 is that all odd pretzel knots for which $s \neq \pm 1$ are not slice, which is particularly easy to read off from the $k$-tuple defining the knot. For odd, 5-stranded pretzel knots this tells us that if all or all but one of the twist parameters have the same sign, then $K$ is not slice. 

Unfortunately, the signature alone is necessary but clearly insufficient for determining sliceness in odd pretzel knots for which $s = \pm 1$. For example, the pretzel knot $K = P(-3, -5, -7, 9, 27)$ has vanishing signature, but the Pretzel Slice-Ribbon Conjecture gives us reason to think that $K$ may not be slice. Such occurrences in the odd, 3-stranded case prompted Jabuka and Greene to turn to an obstruction based on Donaldson's Diagonalization Theorem, which is ultimately phrased as a lattice embedding condition necessary for sliceness. This same obstruction was originally used by Paolo Lisca in \cite{lisca} to classify slice knots within the family of 2-bridge knots.

The use of Donaldson's Diagonalization Theorem to define a `Lattice Embedding Condition' for sliceness is based on the construction of a (potentially hypothetical) closed, definite 4-manifold $X$, created as follows: Assume $K$ is a slice knot. Let $Y$ be the double branched cover of $S^3$ along $K$, $W$ be a rational homology 4-ball with $\partial W = Y$, and $P$ be a canonical\footnote{A canonical definite plumbing $P$ is one for which the weights of the vertices in the corresponding plumbing graph are either all $\geq 2$ or $\leq -2$} 4-dimensional plumbing with $\partial P = Y$. Define $X = P \cup_Y W$. The Lattice Embedding Condition arises by applying the Diagonalization Theorem to $X$, for which it is necessary to verify that the intersection form on $X$, $Q_X$, is definite. We do this in Section \ref{donald}.

The Lattice Embedding Condition for sliceness puts great restrictions on the possible $k$-tuples that can define a slice, odd pretzel knot, so it enables us to conclude that all but a very select subset of such knots are not slice.  Unfortunately, the knots that satisfy both the vanishing signature condition and the Lattice Embedding Condition are not easily differentiated from the knots satisfying the signature condition but {\it not} the Lattice Embedding Condition. For example, sliceness is obstructed for $P(-3, -17, 27)$ and $P(-3, -7, -19, 17, 55)$ by the Lattice Embedding Condition, but sliceness is {\it not} obstructed for $P(-3, -17, 29)$ and $P(-3, -7, -19, 19, 55)$. 

For this reason, Jabuka and Greene introduced a third slice obstruction, this time of their own creation, based on the $d$-invariant from Heegaard-Floer theory. It assumes the same construction used above involving $K$, $Y$, $W$, $P$, and $X$, but it boils down to a comparison of two different `counts' obtained by analysis on the homology long exact sequences of the pairs $(X,W)$ and $(P,Y)$. We refer to it here as `Coset Counting Condition I'. Combining the signature obstruction, the Lattice Embedding Condition, and Coset Counting Condition I, Jabuka and Greene were able to prove their full result. With these same tools, we obtain the following partial results for odd, 5-stranded pretzel knots with signature zero:

\begin{theorem}
\label{thm2}
If $K$ is a 0-pair, odd, 5-stranded pretzel knot, then $K$ is not slice.
\end{theorem}

Coset Condition I fails to obstruct sliceness in $t$-pair, odd, $k$-stranded pretzel knots $K$ if $t \geq 1$, $k$ is odd and $\sigma(K) = 0$, so yet another tool is required in our present case. When $k \geq 5$, the increased number of twist parameters introduces complexity not present when $k = 3$, requiring a more refined `count'  than Jabuka and Greene made themselves when implementing the $d$-invariant obstruction. With just a little bit of work we derive a stronger version of Coset Counting Condition I, and uncreatively deem it Coset Counting Condition II. Combining the signature obstruction, the Lattice Embedding Condition, and Coset Counting Condition II, we prove:

\begin{theorem}
\label{thm3}
If $K$ is a 1-pair, odd, 5-stranded pretzel knot without single-twists, then $K$ is not slice.
\end{theorem}

Theorem \ref{thm3} avoids mention of odd, 5-stranded pretzel knots {\it with} single-twists because they behave slightly differently from those without single-twists for the following reason: any strand with exactly one positive or negative half twist can be transposed with an adjacent strand through a flype as in Figure \ref{flype}. Such a move preserves the smooth knot type so, for example, $P(1,3, -5, 1, -7)$ and $P(1,1,3, -5, -7)$ are not only mutants of one another but also members of the same smooth isotopy class. 

\begin{figure}
\includegraphics[height=400pt]{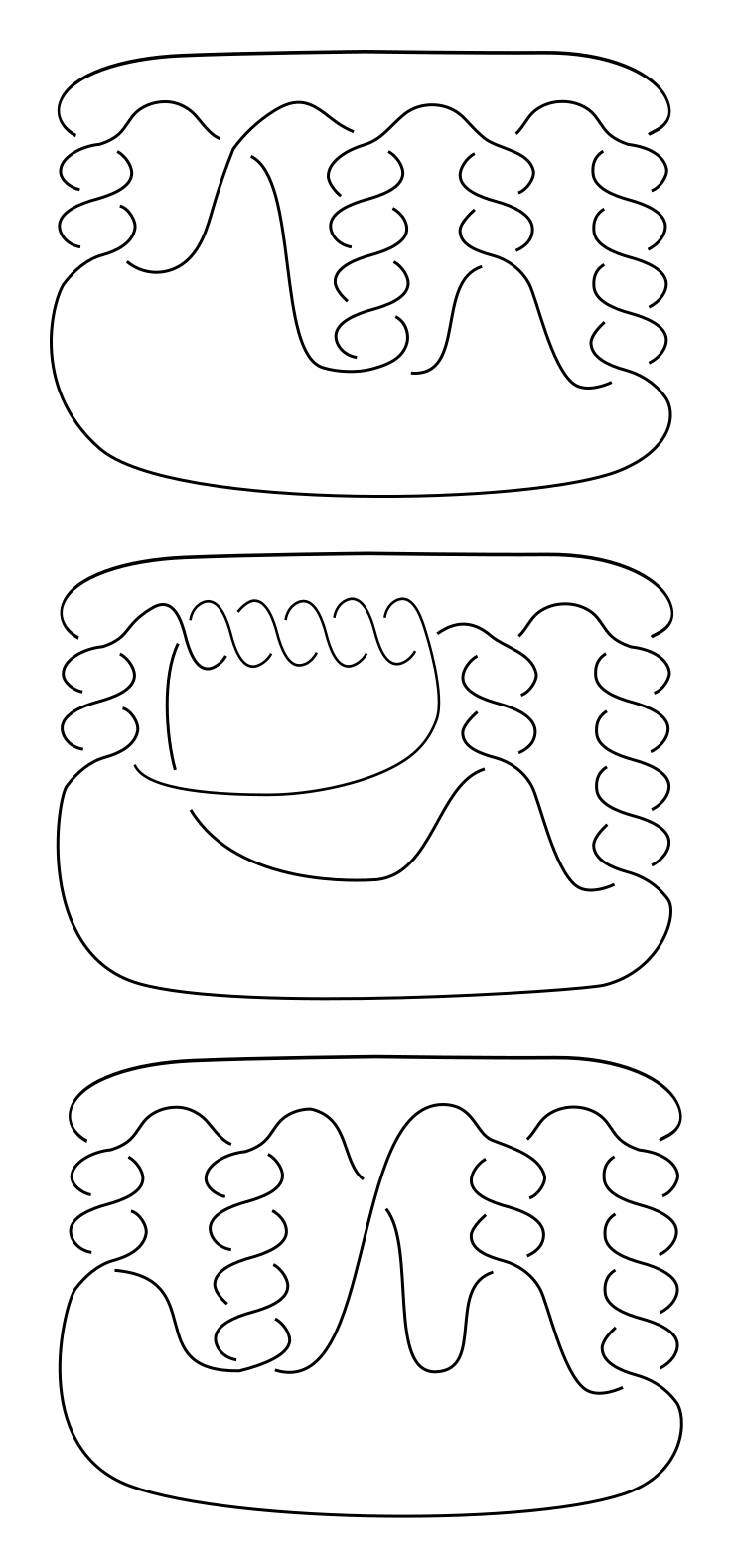}
\caption{Transposition of a single-twist strand, turning $P(3, 1, 5, -3, -5)$ into $P(3,5,1,-3, -3)$.}
\label{flype}
\end{figure}

Furthermore, by flyping we can always ``collect" all strands with $\pm 1$ twists so that they occur in succession. This has the greatest impact on 1- and 2-pair pretzel knots for which at least one of the pairs is $\{-1,1\}$. If $K$ is defined by $\{-1,1,b,c,d\}$, then $K$ is concordant to $P(b,c,d)$ regardless of the initial locations of 1 and -1 in the 5-tuple. It follows that every 2-pair, odd, 5-stranded pretzel knot containing the pair $\{-1,1\}$ is simple ribbon. To contrast, if $K \in P\{-a, a, b,c,d\}$ with $|a| \geq 3$, then $K$ is smoothly concordant to $P(b,c,d)$ if and only if the pair $\{-a,a\}$ is adjacent; it is precisely this fact that leads to $P(3, 5, -3, -5, 7)$ and $P(3,-3, 5,-5,7)$ having different smooth concordance order.

Theorem 2 and Theorem 3 together imply that odd, 5-stranded pretzel knots without single-twists satisfy a weaker version of the Slice-Ribbon Conjecture. Define a knot $K$ to be {\it mutant ribbon} if $K$ is a mutant of a simple ribbon knot. Then we have:

\begin{corollary}
\label{mutantribbon}
If $K$ is a slice, odd, 5-stranded pretzel knot without single-twists, then $K$ is mutant ribbon.
\end{corollary}

For 2-pair, odd, 5-stranded pretzel knots (with or without single-twists) not containing the pair $\{-1,1\}$ and for 1-pairs with single-twists and pair $\{-a,a\} \neq \{-1,1\}$, the signature, the Lattice Embedding Condition, and both Coset Counting Conditions I \& II all fail to obstruct sliceness in the knots that are not simple ribbon because these slice obstructions, at their cores, obstruct the double branched covers of the knots from having certain properties. As previously mentioned, all mutants of a given pretzel knot share the same double branched cover and hence there is no hope of obstructing sliceness for a knot $K \in P\{a,b,c,d,e\}$ if any member of $P\{a,b,c,d,e\}$ is slice. Since 2-pair knots of the form $P(a,-a, b,-b, c,)$ and $P(a,b,-b,-a,c)$ are simple ribbon and therefore slice, we cannot use the aforementioned tools to say that $P(a,b,-a,-b,c)$ is not slice. Similarly, Remark 1.3 in \cite{greene-jabuka} gives that the 1-pair knots with single-twists and pair $\{-a,a\} \neq \{-1,1\}$ of the form $P(a, -a, 1, b,c)$ with $b + c = 4$ are slice, and therefore again there is no way to distinguish between slice and suspected non-slice members of $P\{a,-a, 1,b,c\}$. 

In \cite{HKL}, the authors used twisted Alexander polynomials to show that the 2-pair knot $P(3, 5, -3, -5, 7)$ is not slice, despite being mutant to the simple ribbon knot $P(3,-3,5,-5,7)$. Twisted Alexander polynomials are able to distinguish mutants and, in fact, they can reveal when a knot is not {\it topologically} slice. The issue in using twisted Alexander polynomials to show that odd, 5-stranded pretzel knots satisfy the Slice-Ribbon Conjecture is that these polynomials are difficult to compute. Of the examples computed for pretzel knots to date, there is only one infinite family of pretzel knots whose slice status has been determined using twisted Alexander polynomials. It is the 4-stranded family $K=P(2n,m,-2n\pm 1,-m)$, done by Allison Miller in \cite{miller}.

\noindent {\it Acknowledgements.} I, the author, would like to thank my advisor Paul Melvin for his patience and guidance throughout this project. I am forever indebted to him for not only passing along his knowledge of this subject, but also for his invaluable edits of my work, both in style and content.  

\vspace{.05in}

\section{Framed Links, Weighted Graphs, and Plumbings} \label{graphbasics}
Let $K = P(a_1, ..., a_p, -b_1, ..., -b_n)$ be an odd, $k$-stranded pretzel knot with $k = p + n$ odd, and let $Y$ be the double branched cover of $S^3$ along $K$. As a 3-manifold, we will describe $Y$ by two framed links $L_0$ and $L_+$, which are represented by weighted star-shaped graphs $\Gamma_0$ and $\Gamma_+$, shown in Figure \ref{stargraphs}. In $\Gamma_0$ and $\Gamma_+$, each vertex $v_i$ with weight $w(v_i)$ represents an unknot component $K_i$ with framing $r_i = w(v_i)$; two components $K_i$ and $K_j$ link once in $L_0$ [resp. $L_+$] if the corresponding vertices $v_i$ and $v_j$ share an edge in $\Gamma_0$ [resp. $\Gamma_+$]. 

\begin{figure}
\includegraphics[height=175pt]{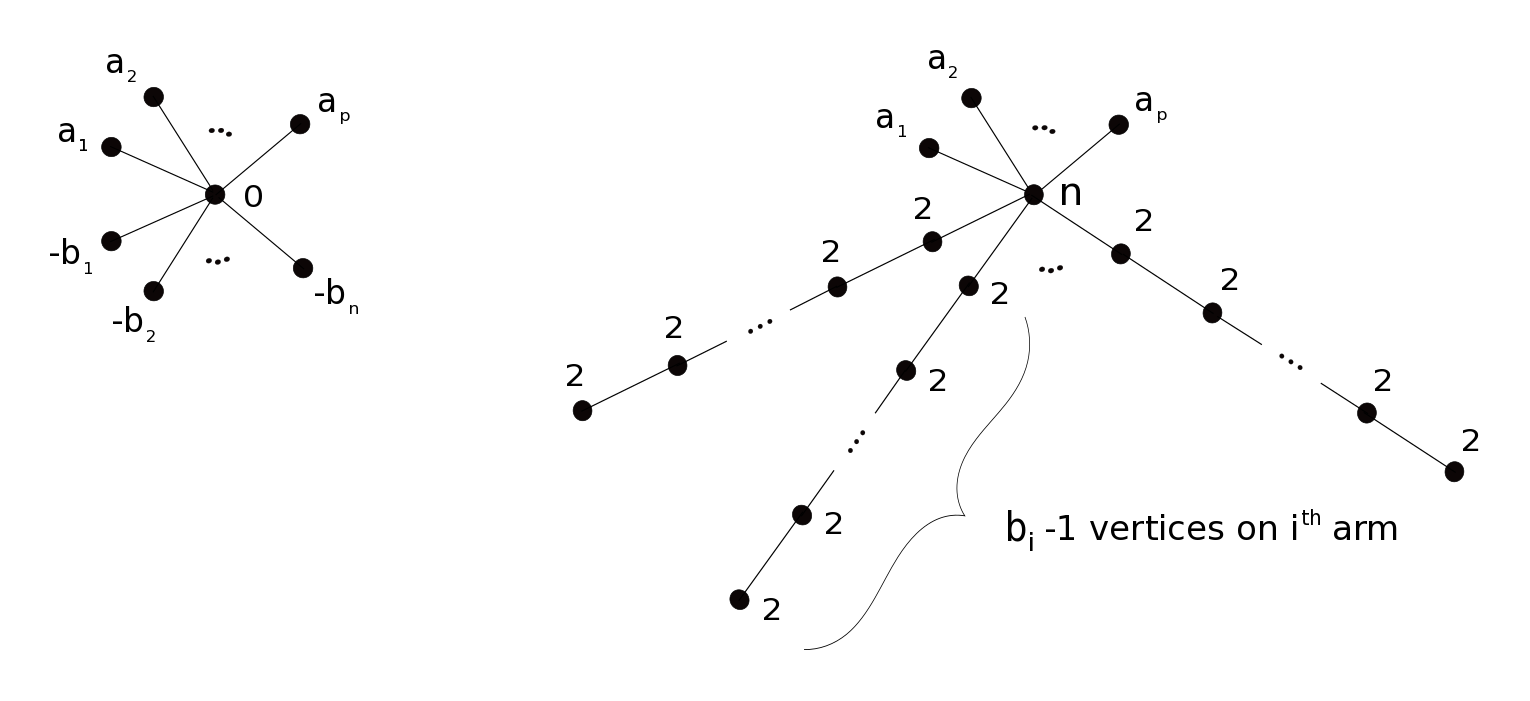}
\caption{Weighted plumbing graphs $\Gamma_{L_0}$ (left) and $\Gamma_{L_+}$ (right).}
\label{stargraphs}
\end{figure}

The links $L_0$ and $L_+$ differ by a sequence of Kirby moves. To transform $L_0$ into $L_+$ via Kirby moves, one operates on the negatively-framed components of $L_0$. The course of Kirby moves needed to change $L_0$ into $L_+$ is described in Figure \ref{kirbycalc} as a sequence of steps that is performed on each negatively-framed component of $L_0$. The accompanying weighted graphs are shown as well. If the prescribed sequence of moves is performed on a component with framing $-b_i$, then Step 5 is repeated $b_i - 3$ times and the original component is ultimately replaced by $b_i - 1$ new components, all unknots with framing 2. In the corresponding weighted graphs, this translates into replacing a single arm of length one, whose lone vertex has weight $-b_i$, by an arm of length $b_i - 1$ containing all weight-2 vertices; the weight of the central vertex increases by 1 for each arm altered.

\begin{figure}
\includegraphics[height=450pt]{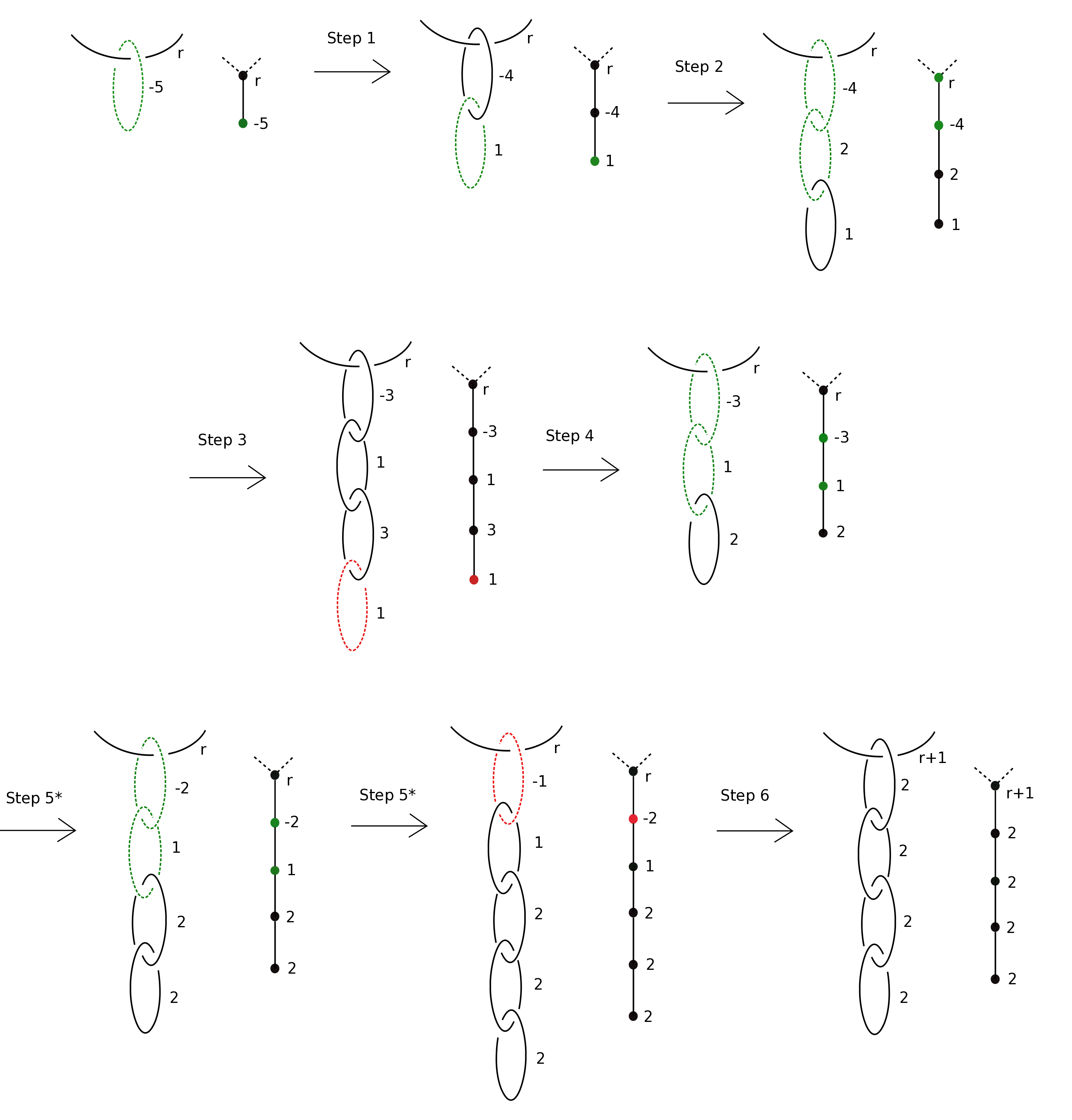}
\caption{Kirby Calculus Sequence $L_0 \rightarrow L_+$}
\label{kirbycalc}
\end{figure}

In addition to describing $Y$, the framed links $L_0$ and $L_+$ and weighted graphs $\Gamma_0$ and $\Gamma_+$ define 4-dimensional plumbings $P_0$ and $P_+$ (respectively) that bound $Y$. Each framed unknot (or weighted vertex) represents a $D^2$ bundle over $S^2$ with Euler class equal to the framing (or weight); linking between components (or an edge between vertices) indicates that the corresponding disk bundles are plumbed together. For more information on the plumbing construction, see \cite{gompf}.

The intersection form $Q_{P_0}$ for $P_0$, when represented as a matrix with basis equal to the set of classes represented by the zero-sections of the above disk bundles, is equal to the incidence matrix for $\Gamma_0$. Similarly, the intersection form of $P_+$, when represented as a matrix using the basis with elements consisting of the homology classes represented by the zero-sections of the above disk bundles, is equal to the incidence matrix for $\Gamma_+$. Knowing the exact sequence of Kirby moves between $L_0$ and $L_+$ allows us to analyze how the signature changes in the underlying 4-manifolds after each step, in particular the overall change in signature from $P_0$ to $P_+$. 

At this point, it is worth detailing a labeling scheme for the vertices of $\Gamma_0$ and $\Gamma_+$ so that the bases for the incidence matrices are ordered consistently. Given the vertex labelings pictured in Figure \ref{vertexlabels}, $\Gamma_0$ will have ordered basis $\{v_0, v_1, ... , v_p, v_{p+1}, ... , v_{p+n}\}$ and $\Gamma_+$ will have ordered basis $\{v_0, v_1, ... , v_p, v_{1,1}, ..., v_{1, b_1 - 1}, ... , v_{n,1}, ... , v_{n, b_n - 1}\}$. Written succinctly, the basis for $\Gamma_+$ can be written $\{v_i, v_{j,r_j}\}$ where $0 \leq i \leq p$, $1 \leq j \leq n$, and $1 \leq r_j \leq j - 1$. It is with these ordered bases for $\Gamma_0$ and $\Gamma_+$ that the matrices for $Q_{P_0}$ and $Q_{P_+}$ are given later.

\begin{figure}
\includegraphics[height=175pt]{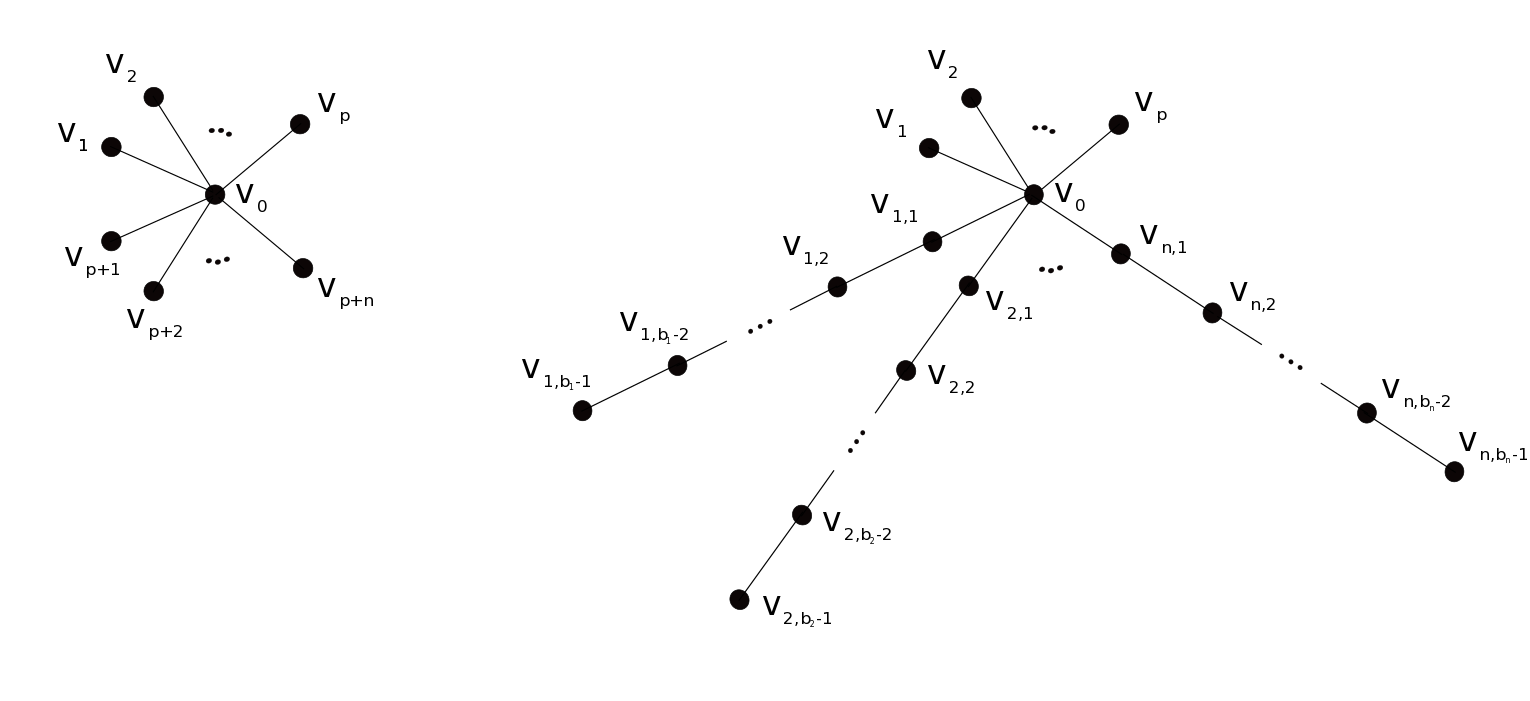}
\caption{Vertex labeling scheme for $\Gamma_0$ (left) and $\Gamma_+$ (right).}
\label{vertexlabels}
\end{figure}

\vspace{.05in}

\section{The Signature Condition and Proof of Theorem 1} \label{signature}
The {\it signature} of a symmetric matrix $Q$, denoted $\sigma(Q)$, is the difference between the number of positive diagonal entries and the number of negative diagonal entries of $Q$, after $Q$ has been diagonalized over $\mathbb{R}$. The {\it signature of a knot} $K$  is defined as $\sigma(K) = \sigma(V^T + V)$, where $V$ is a Seifert matrix for $K$. Given a 4-manifold $X$ with intersection form $Q_X$, the {\it signature of $X$} is defined as the signature of $Q_X$: $\sigma(X) = \sigma(Q_X)$. The signature is an abelian invariant based on the double branched cover of the knot, therefore it cannot detect pretzel mutants. 

The signature is a homomorphism $\sigma : \mathcal{T} \rightarrow \mathbb{Z}$, where $\mathcal{T}$ is the topological knot concordance group. Hence:
\begin{enumerate}
	\item $\sigma(-K) = -\sigma(K)$, where $-K$ is the mirror of $K$, and
	\item $\sigma(K_1 \;  \# \; K_2) = \sigma(K_1) + \sigma(K_2)$.
\end{enumerate}

The signature is also invariant under mutation [c.f. \cite{kearton}]. For pretzel knots, if we combine this fact with (1) we see that computation of the signature of $K = P(p_1, ..., p_k)$ may be obtained using any knot in $P\{p_1, ..., p_k\}$. Often, a specific $K$ is chosen in order to make computations as simple as possible. Homomorphism property (2) implies that if $\sigma(K) > 0$, then the knot $K$ will have infinite order in the topological (and therefore smooth) knot concordance group. A classical theorem relating signatures and sliceness is the following:

\begin{theorem*}
If a knot $K$ is slice, then $\sigma(K) = 0$. 
\end{theorem*}

\noindent The proof of this can be found in \cite{rolfsen}, and with this result we are now ready to prove Theorem \ref{thm:sthm}. 

\noindent {\bf Proof of Theorem 1.}
Let $K = P(a_1, ..., a_p, -b_1, ..., -b_n)$ be an odd pretzel knot. Kauffman and Taylor prove in \cite{Kauff2} that $\sigma(K) = \sigma(Q_{S})$, where $S$ is the double branched cover of $B^4$ along any Seifert surface of $K$. The plumbing manifold $P_0$ described in Section 2 is the double branched cover of $B^4$ along the obvious Seifert surface for $K$, therefore $\sigma(K) = \sigma(Q_{P_0})$. Recall that $P_0$ is described by the graph $\Gamma_0$, and thus $Q_{P_0}$ (with appropriate choice of basis) is given by the incidence matrix of $\Gamma_0$. A straightforward diagonalization of $Q_{P_0}$ shows that:
\[ Q_{P_0} =
\begin{bmatrix}
0 & 1 & 1 & 1 & 1 & 1 & 1\\
1 & a_1 & 0 & 0 & 0 & 0 & 0\\
1 & 0 & \ddots & 0 & 0 & 0 & 0\\
1 & 0 & 0 & a_p& 0 & 0 & 0 \\
1 & 0 & 0 & 0  & -b_1 & 0 & 0\\
1 & 0 & 0 & 0 & 0 & \ddots & 0\\
1 & 0 & 0 & 0 & 0 & 0 & -b_n   
\end{bmatrix}
\sim
\begin{bmatrix}
-\hat{e} & 0 & 0 & 0 & 0 & 0 & 0\\
0 & a_1 & 0 & 0 & 0 & 0 & 0\\
0 & 0 & \ddots & 0 & 0 & 0 & 0\\
0 & 0 & 0 & a_p & 0 & 0 & 0\\
0 & 0 & 0 & 0 & -b_1 & 0 & 0\\
0 & 0 & 0 & 0 & 0 & \ddots & 0\\
0 & 0 & 0 & 0 & 0 & 0 & -b_n 
\end{bmatrix}
\]

Therefore, $\sigma(K) = \sigma(Q_{P_0}) = s - \text{sgn}(\hat{e})$.
\qed

Let $a, b, c,d,e \geq 3$. As mentioned in Section \ref{results}, Theorem \ref{thm:sthm} shows non-sliceness for all odd, 5-stranded pretzel knots $K$ in $P\{a, b, c, d, e\}$ and in $P\{a, b, c, d,-e\}$, since $s$ fails to equal $\pm 1$. But, it also shows non-sliceness for all $K$ in $P\{a, b, c, -d, -e\}$ for which $\frac{1}{a} + \frac{1}{b} + \frac{1}{c} > \frac{1}{d} + \frac{1}{e}$. For example, $K = P(5, 5, 5, -3, -3)$ has nonvanishing signature by Theorem \ref{thm:sthm} and is therefore not slice. 

\vspace{.05in}

\section{Donaldson's Diagonalization Theorem and the Lattice Embedding Condition} \label{donald}

\def\tor{\textup{Tor}}

Donaldson's Diagonalization Theorem constitutes a small piece of the larger topic of Yang-Mills gauge theory. It remains one of the most significant results in the geography problem of 4-manifolds, and it has useful applications to other areas of low-dimensional topology. Donaldson's Diagonalization Theorem can be used to obstruct knot sliceness and it is with this goal in mind that we call on it here. Recall that a closed, oriented 4-manifold $X$ has a unimodular intersection form $Q_X: H_2(X)/\tor \otimes H_2(X)/\tor \rightarrow \mathbb{Z}$\footnote{Here, Tor denotes the torsion part of $H_2(X)$.}, and that $Q_X$ is {\it definite} if $|\sigma(Q_X)| = \text{rk}(Q_X)$. Then:

\begin{theorem*}[Donaldson, 1987]
Let $X$ be a smooth, closed, oriented, 4-manifold with positive definite intersection form $Q_X$. Then $Q_X$ is equivalent over the integers to the standard diagonal form, so in some base:
$$
Q_{X}(u_1,u_2, .\ .\ .\ ,u_r) = u_1^2 + u_2^2 + \cdots + u_r^2
$$
\end{theorem*}  

\begin{remark}
Donaldson's Diagonalization Theorem was originally phrased for $Q_X$ negative definite, making all the $u_i^2$ terms negative. Also, $Q_X$ being definite and diagonalizable means that the pair $(\mathbb{Z}^{b_2(X)},Q_X)$ can be viewed geometrically as a lattice that is isomorphic to $\mathbb{Z}^{b_2(X)}$ with the standard dot product. 
\end{remark}

Donaldson's Diagonalization Theorem is used to obstruct sliceness of a knot in the following way: Assume the knot $K \subset S^3$ is slice and that $Y$ is the 2-fold branched cover of $S^3$ along $K$. Let $P$ be a canonical definite 4-dimensional plumbing manifold satisfying $\partial P = Y$, and let $W$ be the double branched cover of $B^4$ along a slicing disk for $K$. Since $K$ is a knot, $Y$ is a rational homology 3-sphere. Furthermore, $W$ is a rational homology 4-ball with $\partial W = Y$, which follows from the more general fact that the double branched cover of a $\mathbb{Z}/2\mathbb{Z}$-homology ball branched along codimension 2 $\mathbb{Z}/2\mathbb{Z}$-homology ball is again a $\mathbb{Z}/2\mathbb{Z}$-homology ball. A new 4-manifold $X$ is formed by gluing $P$ and $W$ together along their common boundary $Y$ in the usual, orientation-preserving way. This new manifold $X$ will be compact, smooth, oriented, and have definite intersection form, and thus the Diagonalization Theorem applies. This gives that $(\mathbb{Z}^{b_2(X)},Q_X)$ is lattice isomorphic to $(\mathbb{Z}^{b_2(X)},Id)$, the standard $n$-dimensional integer lattice.

The Mayer-Vietoris sequence involving $X = P \cup_Y W$ with rational coefficients shows that $H_2(P)$ includes into $H_2(X)$, and therefore $(\mathbb{Z}^{b_2(P)},Q_P)$ must embed into $(\mathbb{Z}^{b_2(X)},Q_X)$ as a sublattice of full rank. Algebraically, $(\mathbb{Z}^{b_2(P)},Q_P)$ embeds into $(\mathbb{Z}^{b_2(X)},Id)$ if there exists an injection $\alpha:\mathbb{Z}^{b_2(P)} \rightarrow \mathbb{Z}^{b_2(X)}$ such that $Q_P(a,b) = Id(\alpha(a),\alpha(b))$ \cite{greene-jabuka}. If this embedding does not exist then the conclusion is that $X$, as constructed, does not exist. The only assumption made in this construction was that $K$ is slice, therefore the contradiction implies this cannot be the case. Thus, the existence of an embedding $\alpha$ of the lattice $(\mathbb{Z}^{b_2(P)},Q_P)$ into $(\mathbb{Z}^{b_2(X)},Q_X)$ is a necessary condition for the knot $K$ to be slice, which is precisely the obstruction to sliceness utilized in \cite{lisca} and \cite{greene-jabuka}. We call this the Lattice Embedding Condition.%The existence of the embedding $(\mathbb{Z}^{b_2(X)},Q_P) \hookrightarrow (\mathbb{Z}^{b_2(X)},Id)$ will heretofore be referred to as the {\it embedding criterion for sliceness}.

%\begin{figure}
%\includegraphics[height=150pt]{4mfldX.pdf}
%\caption{Gluing schematic of $X$}
%\label{4mfldX}
%\end{figure}

In practice, showing that the embedding $\alpha$ exists amounts to writing down a matrix $A$ for $\alpha$ that satisfies $A^TA = Q_P$. This requires a choice basis for $H_2(P)$ and for $H_2(X)/\tor$. The basis $\{s_i\}$ chosen for $H_2(P)/\tor$ is the set of classes represented by the zero-sections in the disk bundles used to create $P$; the basis $\{e_i\}$ for $H_2(X)/\tor$ is chosen to be one that makes $Q_X$ diagonal by Donaldson's Theorem. As such, each column of $A$ corresponds to one of those 2-spheres in $P$ whose intersection information is recorded by the plumbing graph of $P$. That is, the columns of $A$ must have standard dot products consistent with the information given by the plumbing graph for $P$. 

In an attempt to use Donaldson's Diagonalization Theorem to obstruct sliceness of an odd pretzel knot $K$, we refer back to the Section \ref{graphbasics} and take $P = P_+$, which has plumbing graph $\Gamma_+$ and intersection form $Q_{P_+}$, with matrix equal to the incidence matrix for $\Gamma_+$ with respect to the above bases. By the signature obstruction to sliceness, we need only consider odd pretzel knots $K$ for which $\sigma(K) = 0$. In order to utilize the positive-definite version of Donaldson's Diagonalization Theorem, we need to prove that $Q_X$ is positive definite for $X = P_+ \cup_Y W$. This is done with with the help of the following lemma:

\begin{lemma}
If $K$ is an odd $k$-stranded pretzel knot with $k$ odd and $\sigma(K) = 0$, then either $Q_{P_+}(K)$ or $Q_{P_+}(-K)$ is positive definite.
\label{lem}
\end{lemma}

\begin{proof} 
Let $K \in P\{a_1,...,a_p,-b_1,...,-b_n\}$ be an odd pretzel knot with $\sigma(K) = 0$. By Theorem \ref{thm:sthm}, we know $s = \text{sgn}(\hat{e})$ and we may assume $s = -1$ ($n = p + 1$), {\it after mirroring $K$ if necessary}. The plumbing manifold $P_+$ is represented by $L_+$ and $\Gamma_+$. By choosing the basis for $Q_{P_+}$ to again be the classes represented by the zero-sections of the disk bundles used to create $P_+$, the matrix representative of $Q_{P_+}$ is equal to the incidence matrix of $\Gamma_+$. Since $P_+$ differs from $P_0$ by the prescribed sequence of Kirby moves from Figure \ref{kirbycalc}, performed on each negatively-framed component of $P_0$, we analyze the effect of each move on the signature. 

Each Kirby move used to change $P_0$ into $P_+$ can be decomposed as a combination of the operations $\mathcal{O}_1$ and $\mathcal{O}_2$, which serve as the building blocks for Kirby calculus: Given a framed link $L$, operation $\mathcal{O}_1$ is the addition or subtraction of a single unknotted $S^1$ with framing $\pm 1$ separated from $L$ by an embedded $S^2$ in $S^3$; operation $\mathcal{O}_2$ is a handleslide, in which a component of $L$ can be ``added" to another component at the expense of changing framings on the components involved. The details can be found in \cite{Kirby}. 

Performing operation $\mathcal{O}_2$ on a framed link $L$ will not alter the smooth type of the underlying 4-manifold $P$, and therefore the operation $\mathcal{O}_2$ does not alter the signature $P$. But, operation $\mathcal{O}_1$ {\it does} change $P$ and this change will be reflected in a matrix representative of the intersection form of $P$ as the addition or subtraction of $(\pm 1)$ to $Q_P$ via direct sum. If $(1)$ is added or $(-1)$ is subtracted from $Q_P$, the signature will increase by one; conversely, if $(-1)$ is added or $(1)$ is subtracted from $Q_P$, then the signature will decrease by one. Of course, the handleslide operation $\mathcal{O}_2$ will alter the appearance of $Q_P$, but the change is purely superficial. 

Returning to the Kirby moves from $L_0$ to $L_+$, steps 1, 2, 3, and 5* involve the addition of a $+1$-framed unknot and step 6 involves removing a $-1$-framed unknot; hence, the signature will increase by 1 each time one of these moves performed. Step 4 involves the removal of a $+1$-framed unknot, and therefore decreases the signature by 1 each time it is performed. For each negatively-framed component of $L_0$, each step is performed exactly once with the exception of step 5*, which is repeated $b_i - 3$ times in a sequence performed on a $-b_i$-framed component of $L_0$. Summing over all $n$ negatively-framed components of $L_0$, this implies:
$$
\sigma(Q_{P_+}) = \sigma(Q_{P_0}) + \sum_{i = 1}^n 3 + (b_i - 3) = \sigma(Q_{P_0})  + \sum_{i = 1}^n b_i
$$
\indent Recall that Kauffman's result in \cite{Kauff} implies that $\sigma(Q_{P_0}) = \sigma(K) = 0$. Hence, $\sigma(Q_{P_+}) = \sum_{i = 1}^n b_i$. Based on $L_+$, we can also compute that $\text{rk}(Q_{P_+}) = p + 1 + \sum_{i = 1}^n (b_i - 1) = p + 1 - n + \sum_{i = 1}^n b_i = \sum_{i = 1}^n b_i$, since $p + 1 = n$. Thus, 
$$
|\sigma(Q_{P_+})| = \sum_{i = 1}^n b_i = \text{rk}(Q_{P_+}),
$$
and therefore $Q_{P_+}$ is positive definite.
\end{proof}

With our eye on applying the Diagonalization Theorem to $X$ and the help of Lemma \ref{lem}, we argue that for $X = P_+ \cup_Y W$, $Q_X$ is also positive definite. Consider the following portion of the Mayer-Vietoris sequence for $X$ with rational coefficients:

\begin{center}
\begin{tikzpicture}
\matrix (m) [matrix of math nodes, row sep=3em, column sep=2.5em, text height=1.5ex, text depth=0.25ex]
{0 & \mathbb{Q}^n \oplus 0 & H_2(X; \mathbb{Q}) & 0\\};
\path[->]
(m-1-1) edge (m-1-2)
(m-1-2) edge node[auto] {$i_*$} (m-1-3)
(m-1-3) edge (m-1-4);
\end{tikzpicture}
\end{center}

\noindent The map $i_*$ is an isomorphism which implies that every element $x \in H_2(X)$ is a $\mathbb{Q}$-linear combination of basis elements $\{s_i\}$ for $H_2(P_+)$ and torsion elements of $H_2(W)$. Bilinearity of $Q_X$ and positive-definiteness of $Q_{P_+}$ yield that $Q_X$ is positive definite. Thus, we are free to utilize the previously described construction using Donaldson's Diagonalization Theorem, with $P = P_+$, to obtain the embedding criterion for sliceness on odd, 5-stranded pretzel knots. 

%\vspace{.05in}

%\begin{align*}
%x \; = \; \sum \frac{p_i}{q_i} \; s_i \;\;\; \Longrightarrow \;\;\; & \;\;\; qx \; = \; \sum r_i s_i, \;\; \text{where } q = \prod q_i \text{ and } r_i = \frac{p_i}{q_i} \; q \\
%\Longrightarrow \;\;\; & Q_X(qx, qx) = Q_{P_+}(\sum r_i s_i, \sum r_i s_i) \\
%\Longrightarrow \;\;\; & q^2Q_X(x,x) > 0, \text{since } P_+ \text{ is positive definite} \\
%\Longrightarrow \;\;\; & Q_X(x,x) > 0.
%\end{align*}

In all the results that follow, we use Theorem \ref{thm:sthm} to immediately reduce to considering only those odd, 5-stranded pretzel knots of the form $P(-a, -b, -c, d, e)$ for which sgn$(\hat{e}) = -1$. We use $P(-a, -b, -c, d, e)$ rather than its mirror in order to use the positive definite formulation of Donaldson's Theorem. As stated in the explanation of the Lattice Embedding Condition, we wish to write down a matrix $A$ satisfying $A^TA = Q_{P_+}$. This condition can be phrased as a collection of conditions on the column vectors of $A$: 

\noindent {\bf Embedding Conditions.} For a slice, odd, 5-stranded pretzel knot $K = P(-a, -b, -c, d, e)$, there exist vectors $v_i, v_{j,r} \in \mathbb{Z}^m$, with $m = a + b + c$, satisfying:
\begin{enumerate}
	\item [(0)] $v_0 \cdot v_0 = 3$
	\item $v_1 \cdot v_0 = 1$
	\item $v_2 \cdot v_0 = 1$
	\item $v_1 \cdot v_2 = 0$
	\item $v_1 \cdot v_1 = d$
	\item $v_2 \cdot v_2 = e$
	\item $v_{j,1} \cdot v_0 = 1$
	\item $v_{j,r} \cdot v_{j,r} = 2$
	\item For $r \geq 2$: $v_{j,r} \cdot v_{j,r \pm 1} = 1$
	\item For $r \geq 2$: $v_{j,r} \cdot v_* = 0$, for all vectors $v_* \neq v_{j, r \pm 1}$
\end{enumerate}

The Embedding Conditions impose severe restrictions on the form each $v_i$ and $v_{j,r_j}$ can take. Condition (0) for example, implies that $v_0$ must have exactly three entries equal to $\pm 1$ and zeros otherwise; similarly, condition (7) implies that each vector $v_{j,r_j}$ must have exactly two entries equal to $\pm 1$ and zeros otherwise. It can be verified using conditions $(0)$ - $(7)$ that up to a change of basis, $A$ will have the following form with $\alpha, \beta, \gamma, x,y,z \in \mathbb{Z}$:
$$
\newcommand*{\temp}{\multicolumn{1}{r|}{}}
\left[\begin{array}{ccccccccccccccccccccc } 
v_0 & v_1 & v_ 2 & & v_{a,1} & v_{a,2} & & v_{a,a-1} & & v_{b,1} & v_{b,2} & & v_{b,b-1} & & v_{c,1} & v_{c,2} & & v_{c,c-1}\\
\hline
1 & \alpha & x & \temp & 1 & 0 & \cdots & 0 & \temp & & & & & \temp  & & & & \\
0 & \alpha & x & \temp & -1 & -1 & & 0 & \temp & & & & & \temp & & & & \\
0 & \alpha & x & \temp  & 0 & 1 & & 0 & \temp & & & & & \temp & & & & \\
0 & \alpha & x  & \temp & 0 & 0 & \cdots & 0 & \temp & & & 0 & & \temp & & & 0 & \\
 & \vdots &  & \temp & \vdots & & \ddots & \vdots & \temp & & & & & \temp & & & & \\
0 & \alpha & x  & \temp & 0 & 0 & \cdots & -1 & \temp & & & & & \temp & & & & \\
0 & \alpha & x  & \temp & 0 & 0 & \cdots & 1 & \temp & & & & & \temp & & & & \\
\hline 
1 & \beta & y & \temp & & & & & \temp & 1 & 0 & \cdots & 0 & \temp & & & & \\
0 & \beta & y & \temp & & & & & \temp & -1 & -1 & & 0 & \temp & & & & \\
0 & \beta & y & \temp & & & & & \temp & 0 & 1 & & 0 & \temp & & & & \\
0 & \beta & y & \temp & & & 0 & & \temp & 0 & 0 & \cdots & 0 & \temp & & & 0 & \\ 
 & \vdots & & \temp & & & & & \temp & \vdots & & \ddots & \vdots & \temp & & & & \\
0 & \beta & y & \temp & & & & & \temp & 0 & 0 & \cdots & -1 & \temp & & & & \\ 
0 & \beta & y & \temp & & & & & \temp & 0 & 0 & \cdots & 1 & \temp & & & & \\
\hline
1 & \gamma & z & \temp & & & & & \temp & & & & & \temp & 1 & 0 & \cdots & 0 \\
0 & \gamma & z & \temp & & & & & \temp & & & & & \temp & -1 & -1 & & 0 \\
0 & \gamma & z & \temp & & & & & \temp & & & & & \temp & 0 & 1 & & 0 \\
0 & \gamma & z & \temp & & & 0 & & \temp & & & 0 & & \temp  & 0 & 0 & \cdots & 0\\
 & \vdots & & \temp & & & & & \temp & &  & & & \temp & \vdots & & \ddots & \vdots \\
0 & \gamma & z & \temp & & & & & \temp & & & & & \temp  & 0 & 0 & \cdots & -1 \\
0 & \gamma & z & \temp & & & & & \temp & & & & & \temp & 0 & 0 & \cdots & 1\\
\end{array} \right]
$$

Having $A$ in this explicit form allows us to put restrictions on the unordered 5-tuples $\{a,b,c,d,e\}$ that will satisfy the Embedding Conditions. For fixed $a$, $b$, and $c$, we enumerate the Embedding Conditions in terms of the entries of the column vectors of $A$, which reduces the problem of finding the desired embedding to the problem of finding integers $\alpha, \beta, \gamma, x, y, z$ that satisfy the following system of non-linear equations. Each new condition is numbered to correspond to the original embedding condition that implies it. When referring to the Embedding Conditions by number, no distinction is made between the original conditions and the updated conditions since the updated conditions are direct implications of the originals. 

\noindent {\bf (Updated) Embedding Conditions.} For a slice, odd, 5-stranded pretzel knot $K = P(-a, -b, -c, d, e)$, there exist integers $\alpha, \beta, \gamma, x, y, z$ satisfying: 
\begin{enumerate}
	\item $\alpha + \beta + \gamma = 1$
	\item $x + y + z = 1$
	\item $a \alpha x + b \beta y + c \gamma z = 0$
	\item $a \alpha ^2 + b \beta ^2 + c \gamma ^2 = d$
	\item $a x^2 + b y^2 + c z^2 = e$
\end{enumerate}
\noindent In fact, these updated Embedding Conditions are exactly the contents of Theorem 4.1.6 in \cite{long}, so a more detailed account of these facts can be found there.\footnote{Warning: Long's approach to the problem of sliceness in 5-stranded pretzel knots uses a negative definite convention rather than the positive definite convention of this paper.} 

\vspace{.05in}

\section{The \texorpdfstring{$d$}-Invariants and the Coset Counting Conditions} \label{dinvt}
Peter Ozsv{\'a}th and Zolt{\'a}n Szab{\'o} defined the $d$-invariant $d(Y,\mathfrak{s}) \in \mathbb{Q}$ in the setting of Heegaard-Floer homology for a rational homology 3-sphere $Y$ equipped with a Spin$^c$ structure $\mathfrak{s}$. While the $d$-invariant has an important function as a correction term for the grading in Heegaard-Floer homology, it is highly relevant to 4-manifold topology because it is a Spin$^c$ rational homology bordism invariant. As stated in \cite{AFH}, if $(Y_1, \mathfrak{s}_1)$ and $(Y_1, \mathfrak{s}_1)$ are two pairs such that $Y_i$ is a rational homology 3-sphere and $\mathfrak{s}_i$ is a Spin$^c$ structure on $Y_i$, then if there exists a connected, oriented, smooth cobordism $W$ from $Y_1$ to $Y_2$ with $H_i(W; \mathbb{Q}) = 0$ for $i = 1,2$, which can be endowed with a Spin$^c$ structure $\mathfrak{t}$ whose restriction to $Y_i$ is $\mathfrak{s}_i$, then $d(Y_1, \mathfrak{s}_1) = d(Y_2, \mathfrak{s}_2)$. The proof of this highly nontrivial fact is given in Proposition 9.9 of \cite{AFH}, and it has the following corollary:

\begin{corollary}
\label{dzero}
(Ozsv{\'a}th-Szab{\'o}) Let $Y$ be a rational homology 3-sphere with Spin$^c$ structure $\mathfrak{s}$, and let $W$ be a rational homology 4-ball with $\partial W =Y$ and Spin$^c$ structure $\mathfrak{t}$. If $\mathfrak{s}$ can be extended over $W$ so that $\mathfrak{t}|_Y = \mathfrak{s}$, then $d(Y, \mathfrak{s}) = 0$. 
\end{corollary}

In general $d(Y,\mathfrak{s})$ may be hard to compute, but in \cite{OS} Ozsv{\'a}th and Szab{\'o} give a formula for $d(Y,\mathfrak{s})$ when $Y$ is the boundary of a 4-dimensional plumbing manifold $P$. Their formula holds in more generality than the version presented below, but the formula is stated here in the special case relevant to the present situation. Throughout this section, we refer to $K$, $Y$, $P$, $W$, and $X$ as defined in Section \ref{donald}. To remind the reader of these definitions: $K$ is assumed to be a slice, odd pretzel knot; $Y$ is the double branched cover of $S^3$ along $K$; $W$ is the double branched cover of $B^4$ along a fixed slice disk for $K$ with $\partial W = Y$; $P = P_+$ is a positive definite plumbing manifold with $\partial P = Y$; and $X = P \cup_Y W$ is a closed positive definite manifold. Furthermore, $W$ is a rational homology 4-ball and $Y$ is a rational homology 3-sphere under these assumptions. To state the formula easily and give a more geometric flavor to the material that follows, we first discuss an identification of Spin$^c(Y)$ with $H_1(Y)$.

If $Y$ is a 3-manifold such that $H_1(Y)$ is odd torsion, then there is a natural identification of Spin$^c(Y)$ with $H_1(Y)$. In our current work $Y$ is the double branched cover of $S^3$ along a knot $K$, and a bit of straightforward algebraic topology reveals that $H_1(Y)$ is always odd torsion in this case. The first step in the identification shows a one-to-one correspondence between Spin$^c(Y)$ and vect($Y$), the set of Euler structures on $Y$. An Euler structure on a smooth, closed, connected, oriented 3-manifold $Y$ is an equivalence class of nonsingular tangent vector fields on $Y$, where two vector fields $u$ and $v$ on $Y$ are deemed equivalent if $u$ and $v$ are homotopic as nonsingular vector fields outside of some closed, 3-dimensional ball. This particular identification of Spin$^c(Y)$ with vect($Y$) is due to Vladimir Turaev and constitutes Lemma 1.4 in \cite{turaev}, so the reader is directed there for details. The salient feature of this step is that it allows us to view a Spin$^c$ structure on $Y$ as a vector field over $Y$ under some notion of equivalence. 

Assuming Turaev's identificaiton of Spin$^c(Y)$ with vect($Y$), the second step is to identify vect($Y$) with $H_1(Y)$. If the first Chern class of $Y$ is given by $c_1 ( \cdot) \in H^2(Y)$, then the map sending $v \in \text{vect}(Y)$ to one-half the Poincar{\'e} dual of $c_1(v)$ gives a one-to-one correspondence between vect($Y$) and $H_1(Y)$. Thus, for $\mathfrak{s} \in$ Spin$^c(Y)$ with corresponding $e_{\mathfrak{s}} \in$ vect($Y$), $\mathfrak{s}$ is identified with $\text{PD}(\frac{1}{2}c_1(e_{\mathfrak{s}})) \in H_1(Y)$. 

A second topic necessary to discuss before stating the $d$-invariant formula is that of characteristic elements of $H_2(X)/\tor$, $H_2(P)$, and $H_2(P,Y)$. These definitions involve intersection numbers, and in all cases we will abbreviate the intersection number of two elements $a, b$ in $H_2(X)/\tor$, $H_2(P)$, or $H_2(P,Y)$ by $a \cdot b$ and let the definition of $a \cdot b$ be given by context. As before, $Q_X$ and $Q_P$ are the intersection forms on $X$ and $P$, respectively. We define:
\begin{itemize}
 	\item If $a, b \in H_2(X)/\tor$, then $a \cdot b = Q_X(a,b)$. 
	\item If $a, b \in H_2(P)$, then $a \cdot b = Q_P(a,b)$. 
	\item If $a \in H_2(P,Y)$ and $b \in H_2(P)$, then $a \cdot b = Q_P(x,b)$, where $x = Q_P^{-1}(a) \in H_2(P)$.
	\item If $a,b \in H_2(P,Y)$, then $a \cdot b = Q_P(x,y)$, where $x$ is as above and $y = Q_P^{-1}(b) \in H_2(P)$. 
\end{itemize}

By choosing bases for $H_2(X)/\tor$, $H_2(P)$, and $H_2(P,Y)$, homology classes in these groups can be represented by column vectors and the intersection forms $Q_X$ and $Q_P$ can be represented by matrices. We choose bases as follows: the basis $\{e_i\}$ for $H_2(X)/\tor$ is the one that makes $Q_X$ diagonal by Donaldson's Theorem; the basis $\{s_i\}$ for $H_2(P)$ is the set of homology classes represented by the zero-sections of the disk bundles used to create $P$; lastly, the basis $\{d_i\}$ for $H_2(P,Y)$ is the set of classes represented by single fiber disks in each of the disk bundles of $P$. Note that the fiber disks $\{d_i\}$ are the Hom-duals of the $\{s_i\}$. 

With fixed bases the above intersection numbers can be computed using column vector representatives for homology classes and the matrix representatives for $Q_X$ and $Q_P$. As matrices with the above bases, recall that $Q_P$ is equal to the incidence matrix of the weighted graph representing $P$ and $Q_X$ is equal to the identity matrix of rank $b_2(X)$. By an abuse of notation, we use $Q_P$ to denote both the intersection form for $P$ and its matrix representative in this case. This allows us to write and compute the above intersection numbers in terms of column vectors $a, b$ as follows:  
\begin{itemize}
 	\item If $a, b \in H_2(X)/\tor$, then $a \cdot b = a^Tb$. 
	\item If $a, b \in H_2(P)$, then $a \cdot b = a^TQ_Pb$. 
	\item If $a \in H_2(P,Y)$ and $b \in H_2(P)$, then $a \cdot b = x^TQ_Pb$, where $x = Q_P^{-1}(a) \in H_2(P)$. This simplifies to $a \cdot b = a^T b$.
	\item If $a,b \in H_2(P,Y)$, then $a \cdot b = Q_P(x,y)$, where $x$ is as before and $y = Q_P^{-1}(b) \in H_2(P)$. This simplifies to $a \cdot b = a^T Q_P^{-1} b$.
\end{itemize}

\def\char{\textup{Char}}

Now, we say that an absolute class $w \in H_2(X)/\tor$ is a {\it characteristic class of $X$} if $w \cdot x \equiv x \cdot x$ (mod 2), for all $x \in H_2(X)/\tor$; we say a characteristic class $w$ is {\it minimal} if $w \cdot w \leq z \cdot z$ for all characteristic classes $z$. Characteristic and minimal characteristic elements of $H_2(P)$ are defined similarly. A relative class $w \in H_2(P,Y)$ is {\it characteristic in $X$ with respect to $\mathfrak{s}$}, where $\mathfrak{s}$ is regarded as an element of $H_1(Y)$, if $\partial w = \mathfrak{s}$ and $w \cdot u \equiv u \cdot u$ (mod 2), for all $u \in H_2(P)$. The set of characteristic elements in $H_2(P,Y)$ relative to $\mathfrak{s}$ is denoted by $\char_{\mathfrak{s}}(P)$, which makes an appearance in the formula below. 

We are now ready to state Ozsv{\'a}th and Szab{\'o}'s formula for $d(Y,\mathfrak{s})$ in the case that $Y$ bounds a certain type of 4-dimensional plumbing:

\begin{theorem} [Ozsv{\'a}th and Szab{\'o}]
\label{dFormula}
Let $P$ be a 4-dimensional plumbing with positive definite intersection form $Q_P$, such that the weighted graph of $P$ has at most two vertices whose weights are less than their valences. Then under the identification Spin$^c(Y) \rightarrow H_1(Y)$:
\begin{equation} \label{eq:eq0}
d(Y,\mathfrak{s}) = \underset{w \in \char_{\mathfrak{s}}(P)}{\min} \frac{w \cdot w - \sigma(P)}{4}.
\end{equation}
\end{theorem}

In \cite{greene-jabuka}, Greene and Jabuka use Theorem \ref{dFormula} and Corollary \ref{dzero} to give an obstruction to sliceness for odd pretzel knots through some analysis of the cohomology long exact sequences of the pairs $(P,Y)$ and $(X,W)$. Here, we derive their results in terms of homology and obtain the following commutative diagram. In the diagram the horizontal maps arise from the long exact sequences of the pairs $(P,Y)$ and $(X,W)$; the vertical maps $r$ and $\gamma$ are induced by inclusions;  $\beta$ is an isomorphism due to excision; and $q$ is the usual quotient map. 

\begin{center}
\begin{tikzpicture}
\matrix (m) [matrix of math nodes, row sep=3em, column sep=2.5em, text height=1.5ex, text depth=0.25ex]
{0&H_2(P)&H_2(P,Y)&H_1(Y)&0&0\\
\cdots&H_2(X)&H_2(X,W)&H_1(W)&H_1(X)&0\\
\cdots&H_2(X)/\tor& & & & \\};
\path[->]
(m-1-1) edge (m-1-2)
(m-1-2) edge node[auto] {$\lambda$} (m-1-3)
(m-1-3) edge node[auto]  {$\partial$} (m-1-4)
(m-1-4) edge (m-1-5)
(m-1-5) edge (m-1-6)
(m-1-2) edge node[auto] {$r$} (m-2-2)
(m-1-3) edge node[auto] {$\cong$} (m-2-3)
(m-1-3) edge node[left] {$\beta$} (m-2-3)
(m-1-4) edge node[auto] {$\gamma$} (m-2-4)
(m-1-5) edge (m-2-5)
(m-2-1) edge (m-2-2)
(m-2-2) edge node[auto] {$g$} (m-2-3)
(m-2-2) edge node[auto] {$q$} (m-3-2)
(m-2-3) edge (m-2-4)
(m-2-4) edge (m-2-5)
(m-2-5) edge (m-2-6)
(m-3-1) edge (m-3-2)
(m-3-2) edge node[auto] {$\mu$} (m-2-3);

\end{tikzpicture}
\end{center}

Because $H_2(P)$ is free and $r$ is a homomorphism, the image of $r$ lies entirely in the free part of $H_2(X)$. After letting $\alpha = qr$ and $\alpha^* = \beta^{-1}\mu$, this allows us to use the first isomorphism theorem to eliminate $H_2(X)$ from the diagram. By commutativity, $\lambda$ can be seen to have the factorization $\lambda = \alpha^*\alpha$, converting the previous diagram into:

\begin{center}
\begin{tikzpicture}
\matrix (m) [matrix of math nodes, row sep=3em, column sep=2.5em, text height=1.5ex, text depth=0.25ex]
{0&H_2(P)&H_2(P,Y)&H_1(Y)&0&0\\
\cdots&H_2(X)/\tor&H_2(X,W)&H_1(W)&H_1(X)&0\\};
\path[->]
(m-1-1) edge (m-1-2)
(m-1-2) edge node[auto] {$\lambda$} (m-1-3)
(m-1-3) edge node[auto]  {$\partial$} (m-1-4)
(m-1-4) edge (m-1-5)
(m-1-5) edge (m-1-6)
(m-1-2) edge node[left] {$\alpha$} (m-2-2)
(m-1-3) edge node[auto] {$\cong$} (m-2-3)
(m-1-3) edge node[left] {$\beta$} (m-2-3)
(m-1-4) edge node[auto] {$\gamma$} (m-2-4)
(m-1-5) edge (m-2-5)
(m-2-1) edge (m-2-2)
(m-2-2) edge node[auto] {$\alpha^*$} (m-1-3)
(m-2-2) edge node[auto] {$\mu$} (m-2-3)
(m-2-3) edge node[auto] {$\nu$} (m-2-4)
(m-2-4) edge (m-2-5)
(m-2-5) edge (m-2-6);
\end{tikzpicture}
\end{center}

To use this diagram in conjunction with the Lattice Embedding Condition, it is advantageous to work with matrix representatives of the maps $\alpha$, $\alpha^*$, and $\lambda$. We choose the bases for $H_2(P)$, $H_2(P,Y)$, and $H_2(X)/\text{\tor}$ as before, we let $A$ be the matrix representative for the map $\alpha$ (induced by the embedding of $P$ into $X$), and we let $A^*$ be the matrix for $\alpha^*$. 

The columns of $A$ express the basis elements $\{e_i\}$ of $H_2(X)/\tor$ in terms of the basis disks $\{d_i\}$ for $H_2(P,Y)$. Consequently, the {\it rows} of $A^T$ express the spheres $\{s_i\}$ in terms of the $\{e_i\}$, which implies that the $ij^{th}$ entry in $A^TA$ gives the intersection number between the spheres $s_i$ and $s_j$. Thus $A^TA$ is the matrix of $Q_P$ with respect to the basis $\{s_i\}$. 

Recall that each basis element $d_i$ of $H_2(P,Y)$ is the Hom-dual of the basis element $s_i$ of $H_2(P)$, therefore $\lambda(s_i) = \sum_j(s_i \cdot s_j)d_j$. This implies that with respect to the chosen bases, $\lambda$ (as a linear map from $H_2(P)$ to $H_2(P,Y)$) is represented by the same matrix as is $Q_P$ (regarded as a bilinear map from $H_2(P) \times H_2(P)$ to $\mathbb{Z}$). Namely, $\lambda$ is also represented by $A^TA$. Given that $\lambda = \alpha^*\alpha$, it follows that $A^*A = A^TA$ as matrices. Since $Q_P$ is invertible over $\mathbb Q$, so is $A$; whence $A^*=A^T$. By continuing the abuse in notation whereby we use $Q_P$ to denote both the intersection form on $P$ and its matrix representative in this case, we let $Q_P$ as a matrix represent $\lambda$ with respect to the chosen bases. 

Dropping the less relevant maps, the previous commutative diagram becomes:

\begin{center}
\begin{tikzpicture}
\matrix (m) [matrix of math nodes, row sep=3em, column sep=2.5em, text height=1.5ex, text depth=0.25ex]
{0&H_2(P)&H_2(P,Y)&H_1(Y)&0&0\\
\cdots&H_2(X)/\tor&H_2(X,W)&H_1(W)&H_1(X)&0\\};
\path[->]
(m-1-1) edge (m-1-2)
(m-1-2) edge node[auto] {$Q_P$} (m-1-3)
(m-1-3) edge (m-1-4)
(m-1-4) edge (m-1-5)
(m-1-5) edge (m-1-6)
(m-1-2) edge node[left] {$A$} (m-2-2)
(m-1-3) edge node[auto] {$\cong$} (m-2-3)
(m-1-3) edge (m-2-3)
(m-1-4) edge node[auto] {$\gamma$} (m-2-4)
(m-1-5) edge (m-2-5)
(m-2-1) edge (m-2-2)
(m-2-2) edge node[auto] {$A^T$} (m-1-3)
(m-2-2) edge (m-2-3)
(m-2-3) edge (m-2-4)
(m-2-4) edge (m-2-5)
(m-2-5) edge (m-2-6);
\end{tikzpicture}
\end{center}

\noindent We use this to restate and reprove in a homological setting Greene and Jabuka's $d$-invariant obstruction to sliceness in odd pretzel knots:

\begin{theorem}[Greene-Jabuka]
\label{cosets}
Let $K$ be a slice, odd pretzel knot with $Y$, $W$, $P = P_+$, and $X$ as in the above commutative diagram. Then every coset of coker($\alpha$) contains a minimal characteristic class of $H_2(X)/\tor$. 
\end{theorem}

\begin{proof}
Under the assumption that $K$ is slice, it follows that $K$ satisfies the Embedding Conditions and $\sigma(P) = \text{rk}(Q_P) = \text{rk}(Q_X) = b_2(X) := m$. It also follows from Corollary 3 that $d(Y, \mathfrak{s}) = 0$ for every $\mathfrak{s}$ that extends over $W$. In general, the Spin$^c$ structures on a rational homology 3-sphere $Y$ that extend over a rational homology 4-ball $W$ are identified with precisely those elements in $H_1(Y)$ that bound relative homology classes in $H_2(W,Y)$. As such, they are in one-to-one correspondence with the elements of $\text{ker}(\gamma)$, where $\gamma:H_1(Y) \rightarrow H_1(W)$ is induced by inclusion. 

We may apply Theorem \ref{dFormula} to $Y$ since the plumbing graph of $P$ will have exactly one vertex whose weight is less than its valence, namely, the central node. Formula (\ref{eq:eq0}) implies that $d(Y,\mathfrak{s}) = 0$ for all $\mathfrak{s} \in \text{ker}(\gamma)$ if and only if there exists $w \in \char_{\mathfrak{s}}(P)$ such that $w \cdot w = m$. A straight-forward diagram chase shows that for each $w \in \char_{\mathfrak{s}}(P)$ there exists an element $x \in H_2(X)/\tor$ such that $\alpha^*(x) = w$. We will use this to show that the characteristic classes of $P$ relative to $\mathfrak{s}$ correspond to absolute characteristic classes of $X$ with equal intersection number. This will allow us to compute $w \cdot w$, which appears in formula (\ref{eq:eq0}), by using $x \cdot x$ instead. 

Fix the bases for $H_2(P)$, $H_2(P,Y)$, and $H_2(X)/\tor$ as before, and again let $A = (a_{ij})$ be the matrix representative of $\alpha$ with respect to these bases. Let $w \in \char_{\mathfrak{s}}(P)$ and $x = (x_1, ..., x_m) \in H_2(X)/\tor$ such that $\alpha^*(x) = w$. Recall that $\alpha^*$ is represented by the matrix $A^T$ with respect to these bases. To show that $x$ is characteristic in $X$, it suffices to show that $x \cdot e_j \equiv e_j \cdot e_j$ (mod 2), for all basis elements $e_j$. Since $e_j \cdot e_j = 1$, we need only show that $x \cdot e_j$ -- that is, the $j^{th}$ component of $x$ -- is odd for all $j$. Stated differently, we must show that every component of $x$ is odd. 

By definition of $\char_{\mathfrak{s}}(P)$, $w \cdot u \equiv u \cdot u$ (mod 2) for all $u \in H_2(P)$ and in particular, this holds when $u$ is equal to a basis element $s_j$ for $H_2(P)$: $w \cdot s_j \equiv s_j \cdot s_j$ (mod 2). Observe that for all $j$: 
$$
w \cdot s_j = A^Tx \cdot s_j = x^TAs_j = \sum_i x_i a_{ij},
$$ 
and
$$
s_j \cdot s_j = (Q_P)_{jj} = (A^TA)_{jj}  = \sum_i a_{ij}a_{ij} \equiv \sum_i a_{ij} \;\; \text{ (mod 2)}.
$$
Hence, $\sum_i x_i a_{ij} \equiv \sum_i a_{ij}$ (mod 2) for all $j$. Letting $x_i \equiv 1$ (mod 2) yields a solution to this equation, and in fact is the unique solution since $A$ (mod 2) is invertible. Given that this holds for all $j$, we have shown that $x$ has all odd entries and is therefore characteristic in $X$. Furthermore, since $w \in H_2(P,Y)$, from above we know $w \cdot w = w^TQ_P^{-1}w$. Making the substitutions $Q_P = A^TA$ and $w = A^Tx$ shows that $w \cdot w = x \cdot x$. 

In addition, the diagram chase from before shows that $\text{ker}(\gamma) \cong \text{coker}(\alpha)$. Combining this with the preceding information implies that $d(Y,\mathfrak{s}) = 0$ for all $\mathfrak{s} \in \text{ker}(\gamma)$ with corresponding $k \in \text{coker}(\alpha)$ if and only if there exists $w \in \char_{\mathfrak{s}}(P)$ and $x \in \char(X)$ such that $w = A^Tx$, $x \cdot x = m$, and $x + \text{Im}(\alpha) = k$. Clearly, $x \cdot x = m$ only if $x_i = \pm 1$ for all $i$, which implies that $x$ is a {\it minimal} characteristic class of $X$. Hence, $K$ slice implies that every element of $\text{coker}(\alpha)$, i.e. every coset of Im$(\alpha)$, contains a minimal characteristic class of $X$.
\end{proof}

Theorem \ref{cosets} gives a necessary condition for sliceness for odd, 5-stranded pretzel knots, which can be rephrased in a simpler, more geometric way by analyzing the quotient coker($\alpha$)$= (H_2(X)/\text{\tor}) / \text{Im}(\alpha)$. We will reduce the problem of finding minimal characteristic vectors in each coset of coker$(\alpha)$ to a more visualizable problem of finding lattice points in $\mathbb{Z}^2$ with certain properties.

Since $H_2(X)/\text{\tor} \cong \mathbb{Z}^m$, coker($\alpha$) $\cong \mathbb{Z}^m/\text{Im}(\alpha)$. Given that the image of $\alpha$ with our chosen bases is equal to the span of the columns of $A$, coker($\alpha$) is isomorphic to the quotient of $\mathbb{Z}^m$ by the columns of $A$. Let $\mathcal{U} = \{v_{j,r_j}\}$ be the set of column vectors of $A$ with standard dot product 2, where $1 \leq j \leq n$ and $1 \leq r_j \leq j-1$. Then the columns of $A$, as vectors, are given by $\{v_0, v_1, v_2, \; \mathcal{U} \}$.

Define $B: \mathbb{Z}^m \rightarrow \mathbb{Z}^2$ by:
$$
(x_1,...,x_a,y_1,...,y_b,z_1,...,z_c)^T \mapsto \left( \displaystyle \sum_{i=1}^{a} x_i - \displaystyle \sum_{k=1}^{c} z_k, \displaystyle \sum_{j=1}^{b} y_j - \displaystyle \sum_{k=1}^{c} z_k \right)^T.
$$
It is easy to see that ker($B$) = $\langle v_0, \mathcal{U} \rangle$ and that $B$ is onto, so by the first isomorphism theorem $\mathbb{Z}^2 \cong \mathbb{Z}^m/ \langle v_0, \mathcal{U} \rangle$. It follows that:
$$
\text{coker}(\alpha) \;\;\; \cong \;\;\; \mathbb{Z}^m \Bigm/ \langle v_0, v_1, v_2,\;  \mathcal{U} \rangle 
			%& \cong \left((\mathbb{Z}^m / \langle \mathcal{U} \rangle) \bigm/ \langle \text{C}(v_0) \rangle \right) \Bigm/ \langle \text{BC}(v_1), \text{BC}(v_2) \rangle\\
			%& \cong \left( \mathbb{Z}^3 / \langle \text{C}(v_0) \rangle \right) \Bigm/ \langle \text{BC}(v_1), \text{BC}(v_2) \rangle\\
			\;\;\; \cong \;\;\; \mathbb{Z}^2 \Bigm/ \langle \text{B}(v_1), \text{B}(v_2) \rangle.
$$

Let $\widetilde{v_1} := \text{B}(v_1)$ and $\widetilde{v_2} := \text{B}(v_2)$. Using the above isomorphisms, the slice condition in Theorem 5 can now be rephrased to say that every coset in $\mathbb{Z}^2 \Bigm/ \langle \widetilde{v_1}, \widetilde{v_2} \rangle$ must have a representative in B$(\{ \pm 1 \}^n)$. Thus, we analyze $\mathbb{Z}^2 \Bigm/ \langle \tilde{v_1}, \tilde{v_2} \rangle$ and B$(\{ \pm 1 \}^n)$:
$$
\widetilde{v_1} = \text{B}((\alpha, ... , \alpha, \beta, ... , \beta,\gamma, ..., \gamma))^T = (a\alpha - c\gamma, b\beta - c\gamma)^T,
$$
$$
\widetilde{v_2} = \text{B}((x,...,x,y,...,y,z,...,z))^T = (ax - cz, by - cz)^T,
$$
and
$$	
\text{B}(\{ \pm 1 \}^m) = (q-s, r-s),
$$

\noindent where $-a-c \leq q-s \leq a+c$ and $-b-c \leq r-s \leq b+c$.
 
Observe that $\{ \widetilde{v_1}, \widetilde{v_2} \}$ defines a fundamental domain $\mathcal{R} \subset \mathbb{Z}^2$ and B$(\{ \pm 1 \}^n)$ defines a hexagonal region $\mathcal{H} \subset 2\mathbb{Z}^2$, where each point $(x,y) \in \mathcal{H}$ satisfies $-a-c \leq x \leq a+c$ and $-b-c \leq y \leq b+c$. By the above argument, the slice condition that every element of coker($\alpha$) contain a minimal characteristic vector of the form $(\{\pm 1\}^n)$ is equivalent to the condition that every lattice point in $\mathcal{R}$ be able to be translated onto a lattice point in $\mathcal{H}$ by a linear combination of $\widetilde{v_1}$ and $\widetilde{v_2}$. Since $\mathcal{R}$ is a fundamental domain, every lattice point in $\mathcal{R}$ represents a distinct coset in the quotient $\mathbb{Z}^2 \Bigm/ \langle \widetilde{v_1}, \widetilde{v_2} \rangle$. This proves the following:

\noindent {\bf Coset Condition I}: If $P(a,b,c,d,e)$ is a slice, odd, 5-stranded pretzel knot, then $|\mathcal{R}| \leq |\mathcal{H}|$.

It is possible, however, for many points in $\mathcal{H}$ to belong to the same coset in $\mathbb{Z}^2 \Bigm/ \langle \widetilde{v_1}, \widetilde{v_2} \rangle$. Theorem 5 is clearly contradicted if Coset Condition I is not satisfied, but Theorem 5 is also contradicted if $|\mathcal{R}| > |\bar{\mathcal{H}}|$, where $|\bar{\mathcal{H}}|$ is the number of $\mathcal{R}$-cosets in $\mathcal{H}$. This observation is a refinement of Coset Condition I, which the author unimaginatively deems Coset Condition II:

\noindent {\bf Coset Condition II}: If $P(a,b,c,d,e)$ is slice, then $|\mathcal{R}| \leq |\bar{\mathcal{H}}|$.

\vspace{.05in}

\section{Proof of Theorem \ref{thm2}} \label{sec:thm2}
Due to the slightly different nature of pretzel knots with single-twists versus those without, the proof of Theorem \ref{thm2} is divided according to this distinction. We first give two technical lemmas and then show that all 0-pair, odd, 5-stranded pretzel knots {\it without} single-twists are not slice. Then, we give a refinement of one of the lemmas and with this show that all 0-pair, odd, 5-stranded pretzel knots {\it with} single-twists are not slice.

Recall that a knot $K$ is slice if and only if its mirror $-K$ is slice. To achieve a higher degree of computational ease in the proof, the knot $K = P(-a,-b,-c,d,e) \in P\{a,b,c,-d,-e\}$ will be used rather than its mirror $P(a,b,c,-d,-e)$. Lemmas \ref{lem1} and \ref{lem2}, given next, make clear the conditions under which $P(-a,-b,-c,d,e)$ will be a 0-pair pretzel knot. Without loss of generality, we will assume $a \leq b \leq c$ throughout.

\begin{lemma}
\label{lem1}
Let $K \in P\{a,b,c,-d,-e\}$ satisfy the Embedding Conditions. If any two of $\{\alpha, \beta, \gamma\}$ is zero or if any two of $\{x,y,z\}$ is zero, then $K$ is 1- or 2-pair.
\end{lemma}

\begin{proof}
Choose $K = P(-a, -b, -c, d, e)$. By the symmetry of the embedding conditions on $\{\alpha, \beta, \gamma\}$ and $\{x, y, z\}$, it suffices to prove the result for $\{\alpha, \beta, \gamma\}$. Suppose two of $\{\alpha, \beta, \gamma\}$ are zero. Without loss of generality, let $\alpha = \beta = 0$. By Embedding Condition (1), it must be that $\gamma = 1$ and thus $d = c$ by Embedding Condition (4). Consequently, $P(-a,-b,-c,d,e) = P(-a,-b,-c,c,e)$, so $K$ has at least one pair of canceling twist parameters. Hence $K$ is either 1- or 2-pair.
\end{proof}

\begin{lemma}
\label{lem2}
If $K \in P\{a,b,c,-d,-e\}$ is 0-pair and satisfies the Embedding Conditions, then at most one of $\{\alpha, \beta, \gamma, x, y, z\}$ is zero and, without loss of generality, either $d \geq 4a + b$ and $e \geq a + b + c$, or $d,e \geq a + b + c$. 
\end{lemma}

\begin{proof}
Choose $K = P(-a, -b, -c, d, e)$. Lemma \ref{lem1} implies that at most one of $\{\alpha, \beta, \gamma\}$ is zero and at most one of $\{x, y, z\}$ is zero. If none of $\{\alpha, \beta, \gamma, x, y, z\}$ is zero, then Embedding Conditions (4) and (5) imply that $d = a\alpha^2 + b\beta^2 + c\gamma^2 \geq a + b + c$ and $e = ax^2 + by^2 +cz^2 \geq a + b + c$, since $c \geq b \geq a \geq 1$.

We will show that if any of $\{x, y, z\}$ is zero, then either $K$ is not 0-pair or there is a contradiction to $x, y, z \in \mathbb{Z}$. Suppose $\alpha = 0$ and $\beta \neq \gamma \neq 0$. Embedding Conditions (1) and (4) immediately yield $d = b \beta^2 + c \gamma^2 > 4b + c > 4a + b$, while Embedding Condition (3) implies:

\begin{equation}  
 b \beta y = -c \gamma z
\end{equation}

\noindent Note that if either $y = 0$ or $z = 0$, then the other equals 0 as well. This would imply that $P(-a,-b,-c,d,e)$ is not 0-pair by Lemma \ref{lem1}, a contradiction. Thus, both $y$ and $z$ are nonzero. If $x = 0$, Embedding Condition (2) implies that $z = 1-y$. Substituting this into Equation (2) and solving for $y$, we get:

\begin{equation}\label{yform}
y = \dfrac{c \gamma}{c \gamma - b \beta}.
\end{equation}

Since $\alpha = 0$, Embedding Condition (1) implies that $\beta = 1 - \gamma$. If $\gamma = 1$, then $\beta = 0$ as well and by Lemma \ref{lem1}, $P(-a,-b,-c,d,e)$ is not 0-pair. The same follows if $\beta = 1$. Hence, we can assume $\gamma \geq 2$ or $\gamma \leq -1$. Note that $\beta$ and $\gamma$ always have different signs. 

If $\gamma \geq 2$, then $\beta \leq -1$ and thus $-b \beta > 0$. Thus, Equation (3) takes on the form:

\begin{equation}
y = \dfrac{p}{p + q}.
\end{equation}

\noindent where $p,q \in \mathbb{Z}$ and $p,q \geq 3$. Consequently $y$ cannot be an integer, which contradicts the Embedding Conditions. 

If instead $\gamma \leq -1$, then $\beta \geq 2$. In this case, one can take Equation (2) and solve for $z$ instead of $y$, yielding:

\begin{equation}
z = \dfrac{b \beta}{b \beta - c \gamma}
\end{equation}

\noindent By the same argument given for $\gamma \geq 2$, since $\beta \geq 2$ it follows that $z$ cannot be an integer and the Embedding Conditions are again contradicted. Thus if $\alpha = 0$, all of $\{x,y,z\}$ must be nonzero and $e = ax^2 + by^2 +cz^2 \geq a + b + c$ by Embedding Condition (5). 

If $\beta = 0$, the proof follows similarly with $d = a\alpha^2 + c\gamma^2 \geq 4a + c \geq 4a + b$; if $\gamma = 0$, then again the proof follows similarly with $d = a\alpha^2 + b\beta^2 \geq 4a + b$. In all three cases, $e = ax^2 + by^2 +cz^2 \geq a + b + c$.

If none of $\{\alpha, \beta, \gamma, x, y, z\}$ is zero, then by Embedding Conditions (4) and (5) we get $d = a\alpha^2 + b\beta^2 + c\gamma^2 \geq a + b + c$ and $e = ax^2 + by^2 +cz^2 \geq a + b + c$. 
\end{proof}

The proof of Theorem \ref{thm2} will now proceed by contradicting Coset Counting Condition I. First suppose $K$ is a slice, odd, 5-stranded pretzel knot without single-twists, i.e. $3 \leq a \leq b \leq c$. It follows that $\sigma(K) = 0$, so $K \in P\{a, b, c, -d, -e\}$ by Theorem \ref{thm:sthm} and we may assume $K = P(-a, -b, -c, d, e)$. Since $K$ is slice, $K$ also satisfies the Lattice Embedding Condition and Coset Counting Condition I; the fact that $K$ is 0-pair implies that $d \geq 4a + b$ and $e \geq a + b + c$ by Lemma \ref{lem2}. 

With this, we compute $|\mathcal{R}|$ and $|\mathcal{H}|$, where $\mathcal{R}$ and $\mathcal{H}$ are as in Coset Counting Condition I. Given that $\text{ker}(\gamma) \cong \text{coker}(\alpha)$ and $|\text{ker}(\gamma)| = \sqrt{|H_1(Y)|} = \sqrt{|\text{det} (K)|}$, we know $|\mathcal{R}| = \sqrt{|\text{det}(K)|}$. Theorem 1.4 in \cite{Jabuka} gives the following formula for the determinant of odd pretzel knots $P(p_1, ..., p_k)$: 
$$
\text{det}(K) = \sum_{i = 1}^{k} p_1 \cdots p_{i-1}\hat{p_i}p_{i+1} \cdots p_k \;.
$$
Using this with our choice of $K$, we get:
$$
\text{det}(K) = -abcd - abce + abde + acde +bcde \;.
$$
For $|\mathcal{H}|$, recall that $\mathcal{H} = \{(q - s, r - s) | -a \leq q \leq a, -b \leq r \leq b, -c \leq s \leq c\}$. A straightforward computation yields:
$$
|\mathcal{H}| = ab + ac + bc + a + b + c + 1 \;.
$$ 

We will violate Coset Condition I by arguing that $|\mathcal{R}|^2 > |\mathcal{H}|^2$, using the facts that:
\begin{center}
\begin{enumerate}
	\item $3 \leq a \leq b \leq c$,
	\item $d \geq 4a + b$ and $e \geq a + b + c$ or $d,e \geq a + b + c$, and
	\item $ab > a + b + 1/2$ for $a, b \geq 3$.
\end{enumerate} 
\end{center}
\noindent The case in (2) where $d, e \geq a + b + c$ is the content of Theorem 2.0.3 in \cite{long}, thus it is not proven here. Hence, we assume $d \geq 4a + b$ and $e \geq a + b + c$. In comparing $|\mathcal{R}|^2$ and $|\mathcal{H}|^2$, it suffices to show that a lower bound for $|\mathcal{R}|^2$ is greater than $|\mathcal{H}|^2$, after canceling common terms from both. 

First consider $|\mathcal{H}|^2$:
$$
\begin{aligned}
	|\mathcal{H}|^2 & = (ab + ac + bc + a + b + c + 1)^2 \\
		& = L + M + N + S
\end{aligned}
$$
where: 
$$
\begin{aligned}
	L & = a^2b^2 + a^2c^2 + b^2c^2 + 2(a^2bc + ab^2c + abc^2) \\
	M & = 2(a^2b + a^2c + ab^2 +b^2c) \\
	N & = 2c^2(a + b + 1/2) + 6abc + 4(ab + ac) + 3bc \\
	S & = a^2 + b^2 + bc + 2(a + b + c) + 1
\end{aligned}
$$
Note that $L$ consists of all the quartic terms of $|\mathcal{H}|^2$. Next consider $|\mathcal{R}|^2$:
$$
\begin{aligned}
	|\mathcal{R}|^2 & = |\text{det}(K)| \\
		& = |-abcd - abce + abde + acde +bcde| \\
		& = abd(e-c) + bce(d-a) + acde \\
		& > 5a^2b^2 +4a^2c^2 + b^2c^2 + 8a^2bc + 5ab^2c + 4abc^2 + 4a^3(b+c) + b^3(a+c) \\
		& =: E_3.
\end{aligned}
$$

At this point, it follows that $|\mathcal{R}|^2 - L > E_3 - L$, and $|\mathcal{R}|^2 > |\mathcal{H}|^2$ if $E_3 - L > |\mathcal{H}|^2 - L$. By noting that $E_3 - L$ is a strictly increasing, purely quartic function of $a, b, c$ and $|\mathcal{H}|^2 - L$ is a strictly increasing cubic function in the same three variables, the reader may find it believable that $E_3 - L > |\mathcal{H}|^2 - L$ for all values $a, b, c \geq 3$ provided this inequality holds true for the minimal choice of $a = b = c = 3$. As it happens, when $a = b = c = 3$, $E_3 - L = 2298$ and $ |\mathcal{H}|^2 - L = 640$. For the more suspicious reader, we continue with the proof and observe:
$$
\begin{aligned}
	E_3 - L & = 4a(a^2b + a^2c) + b(ab^2 + b^2c) + 4a^2b^2 + 3a^2c^2 + 6a^2bc + 3ab^2c + 2abc^2 \\
		& > 12(a^2b + a^2c) + 3(ab^2 + b^2c) + 4a^2b^2 + 3a^2c^2 + 6a^2bc + 3ab^2c + 2abc^2 \\
		& =: E_2.
\end{aligned} 
$$
Furthermore:
$$
\begin{aligned}
	E_2 - M & = 10(a^2b + a^2c) + (ab^2 + b^2c) + 4a^2b^2 + 3a^2c^2 + 6a^2bc + 3ab^2c + 2abc^2 \\
		& > 2c^2(a + b + 1/2) + 6abc + 4(ab + ac) + 3bc  + 4a^2b^2 + 3a^2c^2 \\
		& \;\;\;\;\; + 3ab^2c + 6(a^2b + a^c) + ab^2 \\
		& =: E_1.
\end{aligned}
$$
Lastly:
$$
\begin{aligned}
	E_1 - N & = 4a^2b^2 + 3a^2c^2 + 3ab^2c + 6(a^2b + a^2c) + ab^2 \\
		& = 6a^2b + ab^2 + 3a^2c^2 + (4a^2b^2 + 3ab^2c + 6a^2c) \\
		& > a^2 + b^2 + bc + 2(a + b + c) + 1 \\
		& = S \\
		& = |\mathcal{H}|^2 - L - M - N
\end{aligned}
$$
This shows that:
$$
 \begin{aligned}
 	|\mathcal{R}|^2 - L - M - N & > E_3 - L - M - N \\
						& > E_2 - M -N \\
						& > E_1 - N \\
						& >  |\mathcal{H}|^2 - L - M - N
\end{aligned}
$$
which implies that $|\mathcal{R}|^2 > |\mathcal{H}|^2$, as desired. This completes the proof that all 0-pair, odd, 5-stranded pretzel knots without single-twists are not slice.

Next, we address the knots with single-twists. As before, we assume all knots in question are slice and therefore satisfy the signature condition, the Lattice Embedding condition, and both Coset Counting Conditions. Just a little bit of thought reveals that possibly after mirroring, the signature condition yields only three cases to consider for 0-pair, odd, 5-stranded pretzel knots with single-twists. For $K \in P\{-a,-b,-c, d,e\}$, the cases are:

\begin{enumerate}
	\item $a = b = c =1$ and $d,e \geq 3$.
	\item $a = b = 1$ and $c,d,e \geq 3$.
	\item $a = 1$ and $b,c,d,e \geq 3$.
\end{enumerate}

\noindent Since the Lattice Embedding Conditions hold, there exist $\alpha$, $\beta$, $\gamma$, $x$, $y$, $z \in \mathbb{Z}$ satisfying the system of equations given in Section \ref{donald}. Thus, we have the same starting point for 0-pair pretzel knots with single-twists as for 0-pair knots without single-twists. Lemma \ref{lem1} and Lemma \ref{lem2} still apply here for all three cases of 0-pair pretzel knots with single-twists. To obstruct sliceness for 0-pair knots $P(-a,-b,-c,d,e)$ with single twists, however, it is necessary to get more precise lower bounds on $d$ and $e$ than are obtained in Lemma \ref{lem2}. 

\begin{lemma}
\label{lem3}
If $K \in P\{-a,-b,-c,d,e\}$ is 0-pair and $d$ is equal to its lower bound (either $d = 4a + b$ or $d = a + b + c$), then $e \geq 4a + 4b + c$. 
\end{lemma}

\begin{proof}
First, suppose $d = 4a + b$ and $e = a + b + c$. By the Embedding Conditions, it follows that $\alpha = 2, \beta = -1$, and $\gamma = 0$, and $|x| = |y| = |z| = 1$. Embedding Condition (3) says $a\alpha x + b \beta y + c \gamma z = 0$, which reduces to $\pm 2a = b$ after substitutions. But, $b$ is odd so we have a contradiction. If instead we suppose that $d = e = a + b + c$, then by the Embedding Conditions $|\alpha| = |\beta| = |\gamma| = |x| = |y| = |z| = 1$. After substitutions, Embedding Condition (3) becomes $c = \pm a \pm b$, which is again a contradiction since all three of $a,b,c$ are odd. 

Hence when $d = 4a + b$ or $d = a + b + c$, $e \neq a + b + c$. In words, both $d$ and $e$ cannot simultaneously achieve their lower bounds as given in Lemma \ref{lem2}. It follows that at least one of $|x|$, $|y|$, or $|z|$ must be $\geq 2$. But, in fact, we can show that at least two of $|x|$, $|y|$, or $|z|$ must be $\geq 2$: Without loss of generality, suppose $x = \pm 2$. Embedding Condition (2), which says $x + y + z = 1$, yields two possibilities: If $x = 2$, then $y + z = -1$; if $x = -2$, then $y + z = 3$. In both cases, it is impossible for both $|y|= 1$ and $|z| = 1$, and therefore $|y| \geq 2$ or $|z| \geq 2$. By the symmetry in $x,y,z$ of Embedding Condition (2), it follows that at least two of $|x|$, $|y|$, or $|z|$ must be $\geq 2$.

The choices of $|x|,|y|,|z|$ that satisfy the above discovery and that minimize $e$ are $|x| = |y| = 2$ and $|z| = 1$, which yields $e = ax^2 + by^2 + cz^2 = 4a + 4b + c$. Thus, if $d$ is equal to a lower bound then $e \geq 4a + 4b + c$.
\end{proof}

We now continue with the proof of Theorem \ref{1thm0}. The goal in each of the following cases is to arrive at a contradiction to Coset Counting Condition I by showing that $|\mathcal{R}|^2 > |\mathcal{H}|$. 

\noindent {\bf Case 1:} $K \in P\{-a,-b,-c, d,e\}$ with $a = b = c = 1$. \\
\indent By Lemma \ref{lem2}, $d \geq 4a + b = 5$ or $d \geq a + b + c = 3$. Assume $d = 3$. By Lemma \ref{lem3}, we know $e \geq 4a + 4b + c = 9$. Then: 
$$
\begin{aligned}
	|\mathcal{R}|^2 & = |\text{det}(K)| \\
		& = |-abcd - abce + abde + acde +bcde| \\
		& = d(e-1) + e(d-1) + de \\
		& \geq 69 \\
		& > 49 \\
		& = (ab + ac + bc + a + b + c + 1)^2 \\
		& = |\mathcal{H}|^2, 
\end{aligned}
$$
as desired.

\noindent {\bf Case 2:} $K \in P\{-a,-b,-c, d,e\}$ with $a = b = 1$ and $c \geq 3$. \\
\indent By Lemma \ref{lem2}, $d \geq 4a + b = 5$ or $d \geq a + b + c = 2 + c$. These two expressions agree if $c = 3$, but otherwise $4a + b < 2 + c$. Thus, assuming $d \geq 5$ accounts for both situations. By Lemma \ref{lem3}, $e \geq 4a + 4b + c = 8 + c$. Then:
$$
\begin{aligned}
	|\mathcal{R}|^2 & = |\text{det}(K)| \\
		& = |-abcd - abce + abde + acde +bcde| \\
		& = d(e-c) + ce(d-1) + cde \\
		& \geq 5(c+8-c) + c(8+c)(5-1) + c(5)(c+8) \\
		& = 9c^2 + 72c +40
		& > 9c^2 + 24c + 16 \\
		& = (ab + ac + bc + a + b + c + 1)^2 \\
		& = |\mathcal{H}|^2, 
\end{aligned}
$$
as desired.

\noindent {\bf Case 3:} $K \in P\{-a,-b,-c, d,e\}$ with $a = 1$ and $b,c \geq 3$. \\
\indent By Lemma \ref{lem2}, $d \geq 4a + b = 4 + b$ or $d \geq a + b + c = 1 + b + c$. Assuming $d \geq b + 4$ accounts for both situations. By Lemma \ref{lem3}, $e \geq 4a + 4b + c = 4 + 4b + c$. Then:
$$
\begin{aligned}
	|\mathcal{R}|^2 & = |\text{det}(K)| \\
		& = |-abcd - abce + abde + acde +bcde| \\
		& = bd(e-c) + bce(d-1) + cde \\
		& \geq b(b+4)(4 + 4b + c - c) + bc(4 + 4b + c)(b+4-1) + c(b+4)(4 + 4b + c) \\
		& = 4b^3c + b^2c^2 + 4b^3 + 20b^2c + 4bc^2 + 20b^2 + 32bc + 4c^2 + 16b + 16c.
\end{aligned}
$$
Also:
$$
\begin{aligned}
	|\mathcal{H}|^2 & = (ab + ac + bc + a + b + c + 1)^2 \\
		& = b^2c^2 + 4b^2c + 4bc^2 + 4b^2 + 12bc + 4 c^2 +8b +8c + 4.
\end{aligned}
$$

\noindent Let $L = b^2c^2 + 4b^2c + 4bc^2 + 4b^2 + 12bc + 4 c^2 +8b +8c$. Then:
$$
\begin{aligned}
	|\mathcal{R}|^2 - L & \geq 4b^3c + b^2c^2 + 4b^3 + 20b^2c + 4bc^2 + 20b^2 + 32bc + 4c^2 + 16b + 16c - L \\ 
		& = 4b^3c + 4b^3 + 16b^2c + 16b^2 + 20bc + 8b + 8c \\
		& > 4 \\
		& = |\mathcal{H}|^2 - L.
\end{aligned}
$$
Thus, $|\mathcal{R}|^2 > |\mathcal{H}|$. This concludes the proof that 0-pair, odd pretzel knots with single-twists are not slice, and therefore all 0-pair, odd, 5-stranded pretzel knots are not slice.

\qed

\section{Proof of Theorem \ref{thm3}} \label{sec:thm3}

It suffices to consider only the single-pair pretzel knots $P(a,b,c,d,e)$ for which the signature vanishes and both the Lattice Embedding Condition and Coset Counting Condition I are satisfied. Let $a,b,c,d,e >0$ such that $a \leq b \leq c$, and assume that $K=P(-a,-b,-c,d,e)$. Let $Y$, $P = P_+$, $W$, $X$, and the embedding map $\alpha: H_2(P) \rightarrow H_2(X)/\tor$ be as before. 

Theorem \ref{cosets} gives that if $K$ is slice, then every element in coker($\alpha$) has a coset representative in the set $\{\pm 1 \}^m$, where $m = a + b + c$. Let $v_1$ and $v_2$ be the second and third columns (respectively) in the matrix $A$ of $\alpha$ with respect to the bases chosen in Sections \ref{donald} and \ref{dinvt}; lastly, let $B$ be the map outlined in Section \ref{dinvt}. Recall from Section \ref{donald} the regions $\mathcal{R}$ and $\mathcal{H}$ in the plane associated with $A$: $\mathcal{R}$ is defined by the vectors $\{ \widetilde{v_1}, \widetilde{v_2} \}$, with $B(v_1) = \widetilde{v_1}$ and $B(v_2) = \widetilde{v_2}$, and $\mathcal{H}$ is the set of lattice points $(x,y) \in 2\mathbb{R}^2$ such that $-a-c \leq x \leq a + c$ and $-b-c \leq y \leq b + c$. The argument now reduces to counting the number of lattice points in $\mathcal{R}$ and in $\bar{\mathcal{H}}$, where $\bar{\mathcal{H}} = \mathcal{H} / \langle \widetilde{v_1}, \widetilde{v_2} \rangle$.

The number of lattice points in $\mathcal{R}$ is computed by finding the determinant of the $2 \times 2$ matrix with column vectors $\widetilde{v_1}$ and $\widetilde{v_2}$: 
$$
| \mathcal{R} | = 
\begin{vmatrix}
\; a\alpha - c\gamma & ax - cz \; \\
\; b\beta - c\gamma & by - cz \;
\end{vmatrix}.
$$
This determinant counts all lattice points strictly on the interior of $\mathcal{R}$ as well as the lattice points lying on one half of the boundary. The number of lattice points in $\mathcal{H}$ is always $ab+ac+bc+a+b+c+1$, which gives the total number of lattice points on the interior of $\mathcal{H}$, together with {\it all} lattice points lying on the boundary of $\mathcal{H}$.
 
For a 1-pair pretzel knot, three cases must be considered: (1) when the pair is $\{a,-a\}$, (2) when the pair is $\{b,-b\}$, and (3) when the pair is $\{c,-c\}$. By assumption, the twist parameters in all three cases satisfy the embedding criterion.

\noindent {\bf Case I:} $K \in P\{-a,-b,-c,a,d\}$ \\
\indent When the twist parameters contain the pair $\{a,-a\}$, it follows that $\alpha = 1$, $\beta = \gamma = x = 0$, and $y$ and $z$ are nonzero. This yields $\widetilde{v_1} = (a, 0)$ and $\widetilde{v_2} = (-cz, by - cz)$, hence:
$$
| \mathcal{R} | = 
\begin{vmatrix}
\; a & - cz \; \\
\; 0 & by - cz \;
\end{vmatrix}
= a\; |by - cz|.
$$

As $y \rightarrow \infty$, it follows that $z \rightarrow -\infty$ by Embedding Condition (2) which says that $1 = x + y + z = 0 + y + z$; thus $|\mathcal{R}| \rightarrow \infty$. Similarly, as $y \rightarrow - \infty$ (and $ z \rightarrow \infty$), $|\mathcal{R}| \rightarrow \infty$. For this reason, $|\mathcal{R}|$ is minimized when $y$ and $z$ are both small in absolute value, i.e. when $\widetilde{v_2}$ is short. Given that $b < c$, $\widetilde{v_2}$ is shortest when $y = 2$ and $z = -1$. In this case:

\begin{equation} \label{R1}
|\mathcal{R}| = a \; |2b + c| = 2ab + ac.
\end{equation}

Now we compute an upper bound for $|\bar{\mathcal{H}}|$. Since $\mathcal{H}$ only depends on $a$, $b$, and $c$, a direct calculation of the exact size and shape of $\mathcal{H}$ is easily obtained for any fixed values of $a, b, c$. Furthermore, the simple form of $\widetilde{v_1} = (a,0)$ allows us to see that many of the lattice points in $\mathcal{H}$ lie in the same $\mathcal{R}$-coset. To (partially) determine $\bar{\mathcal{H}}$, we simply identify all lattice points in $\mathcal{H}$ that differ by some multiple of $(a, 0)$ (horizonal translations). Given that $\widetilde{v_1}$ is only one of the two vectors used to define $\mathcal{R}$, incorporating $\widetilde{v_2}$ would only create further collapsing among the cosets, therefore the actual value of $|\bar{\mathcal{H}}|$ is less than or equal to number of cosets computed as above. Important note: the Figures \ref{sp1}, \ref{sp3}, and \ref{sp5} show the lattice points in $\mathcal{H} \subset 2\mathbb{Z}^2$. That is, each grid square is $2 \times 2$.

\begin{figure}
\includegraphics[height=350pt]{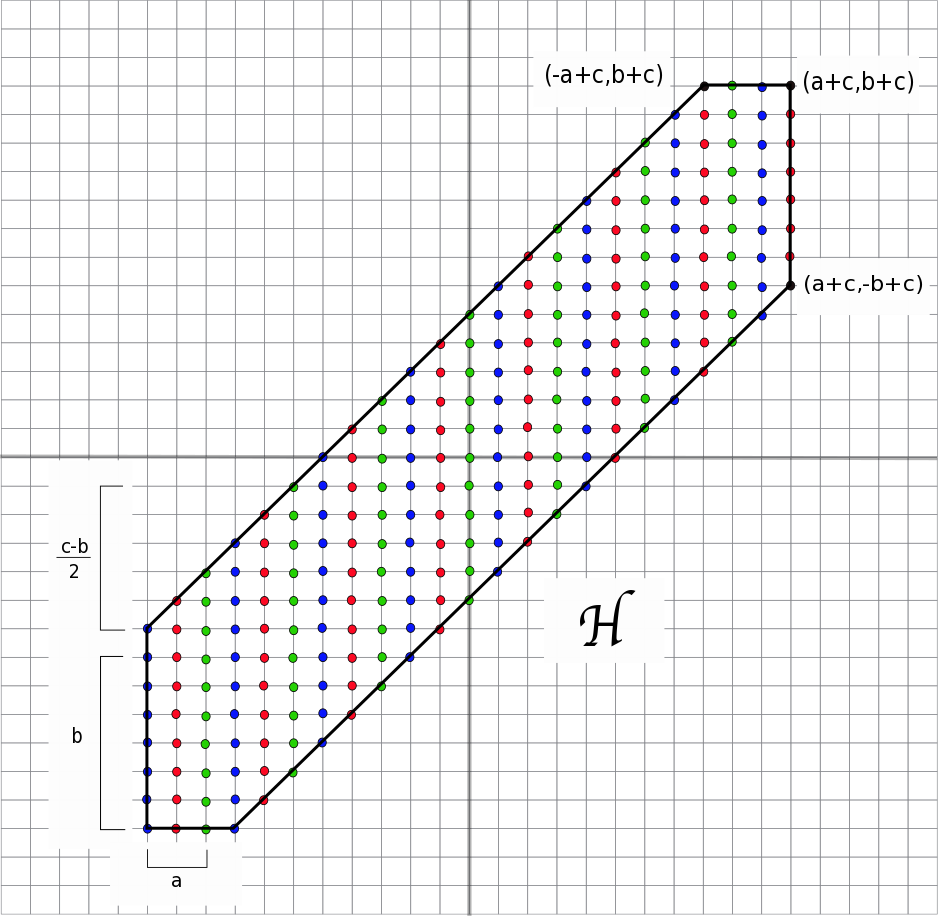}
\caption{$\mathcal{H}$ and its $\mathcal{R}$-cosets for $P(-3,-7, -19, 3, 47)$. Points of the same color {\it with the same $y$-coordinate} represent the same $\mathcal{R}$-coset.}
\label{sp1}
\end{figure} 

%\begin{figure}
%\includegraphics[height=400pt]{SinglePairA2.png}
%\caption{$\mathcal{H}$ and its $\mathcal{R}$-cosets for $P(-3,-7, -19, 3, 47)$, repeat representatives removed.}
%\label{sp2}
%\end{figure} 

From Figure \ref{sp1}, we see that each of the $b+c + 1$ rows in $\mathcal{H}$ always has $a$ distinct $\mathcal{R}$-cosets. Thus, an upper bound for $|\bar{\mathcal{H}}|$ is given by:
$$
\begin{aligned}
	|\bar{\mathcal{H}}|	& \leq a(b + c + 1) \\
			& = ab + ac + a.
\end{aligned}
$$
Comparing this to (\ref{R1}), the desired result is achieved:
$$
\begin{aligned}
	|\bar{\mathcal{H}}|	& \leq ab + ac + a\\
			& < 2ab + ac\\
			& = |\mathcal{R}|,
\end{aligned}
$$
since $a \leq b \leq c$. Hence, a 5-stranded odd pretzel knot $K \in P\{-a,-b,-c,a,d\}$, with $a, b, c, d \geq 3$, is not slice.

\noindent {\bf Case 2:} $K \in P\{-a,-b,-c,b,d\}$ \\
\indent When the twist parameters contain the pair $\{b,-b\}$, it follows that $\beta = 1$, $\alpha = \gamma = y = 0$, and $x$ and $z$ are nonzero. With this, $\widetilde{v_1} = (0,b)$ and $\widetilde{v_2} = (-cz, by - cz)$, hence:
$$
| \mathcal{R} | = 
\begin{vmatrix}
\; 0 & ax - cz \; \\
\; b & - cz \;
\end{vmatrix}
= b\; |ax - cz|.
$$
Following the logic of Case 1, it suffices to show that $|\mathcal{R}| > |\bar{\mathcal{H}}|$ when the length of $\widetilde{v_2}$ is minimized. Since $a < c$, $\widetilde{v_2}$ is shortest when $x = 2$ and $z = -1$, so:

\begin{equation} \label{R2}
|\mathcal{R}| = b\; |2a + c| = 2ab + bc. 
\end{equation}

The upper bound for $|\bar{\mathcal{H}}|$ is computed for Case 2 in a similar manner as for Case 1, the only difference being that lattice points in $\mathcal{H}$ are identified as being in the same coset when they differ by multiple of $\widetilde{v_1} = (0,b)$ (vertical translations). Each of the $a+c + 1$ columns in $\mathcal{H}$ always has $b$ distinct $\mathcal{R}$-cosets. Thus, an upper bound for $|\bar{\mathcal{H}}|$ is given by:

\begin{figure}
\includegraphics[height=350pt]{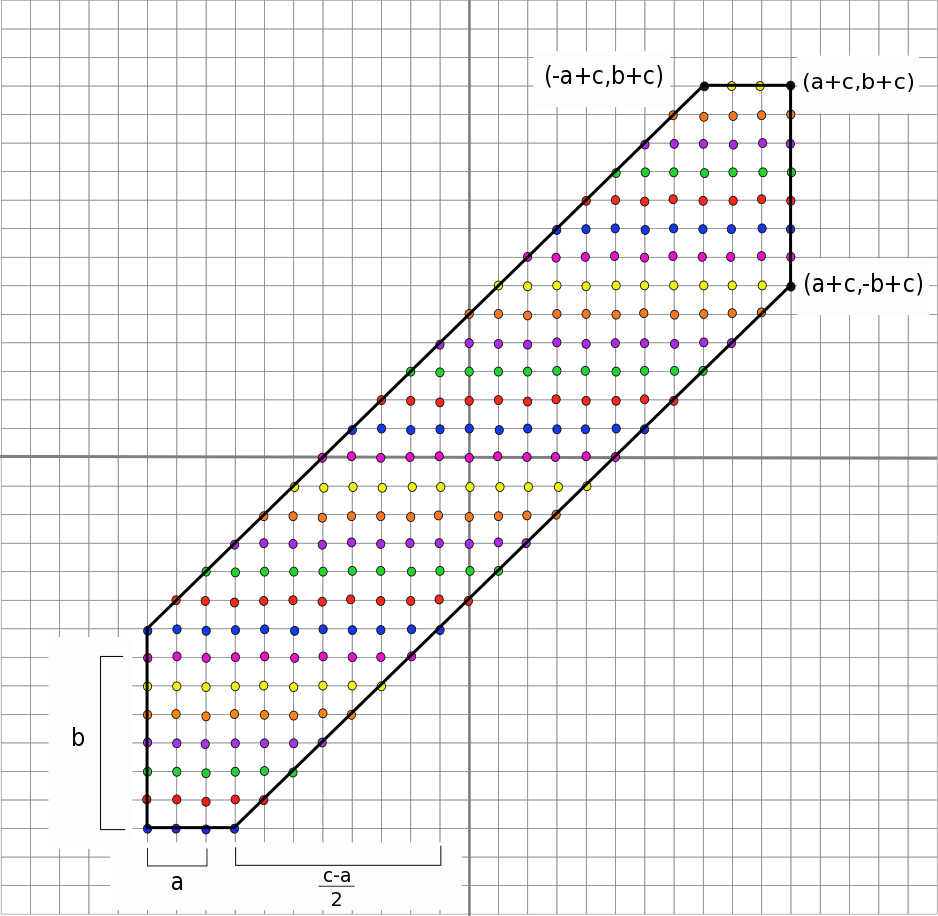}
\caption{$\mathcal{H}$ and its $\mathcal{R}$-cosets for $P(-3,-7, -19, 7, 31)$. Points of the same color {\it with the same $x$-coordinate} represent the same $\mathcal{R}$-coset.}
\label{sp3}
\end{figure} 

%\begin{figure}
%\includegraphics[height=400pt]{SinglePairB2.png}
%\caption{$\mathcal{H}$ and its $\mathcal{R}$-cosets for $P(-3,-7, 19,7,31)$, repeat representatives removed.}
%\label{sp4}
%\end{figure}

$$
\begin{aligned}
	|\bar{\mathcal{H}}|	& \leq b(a + c + 1) \\
			& = ab + bc + b.
\end{aligned}
$$
Comparing this to (\ref{R2}), again the desired result is achieved:
$$
\begin{aligned}
	|\bar{\mathcal{H}}|	& \leq ab + bc + b\\
			& < 2ab + bc\\
			& = |\mathcal{R}|,
\end{aligned}
$$
since $b \geq a \geq 3$. Hence, 5-stranded pretzel knots $K \in P\{-a,-b,-c,b,d)\}$, with $a, b, c,d \geq 3$, are not slice.

\noindent {\bf Case 3:} $K \in P\{-a,-b,-c,c,d\}$ \\
\indent The final case of 5-stranded odd single-pair pretzel knots has $(c,-c)$ as the pair in the twist parameters. Unlike for the 1-pair cases where the pair of canceling twist parameters is $\{a, -a\}$ or $\{b, -b\}$, the case with $\{c, -c\}$ does not necessarily imply that $\gamma = 1$, $\alpha = \beta = z = 0$, with $x$ and $y$ nonzero. Since $c \geq b \geq a$, it is possible that both $\alpha$ and $\beta$ are non-zero and Embedding Condition (4) is satisfied by $c = a\alpha^2 + b\beta^2$. In this case, however, the proof of Lemma 3 shows we would have $c \geq 4a + b$ and $e = ax^2 + by^2 +cz^2 \geq a + b + c$, which implies that $P(-a, -b, -c, c, d)$ is not slice by the proof of Theorem \ref{thm2}. 

Hence, we need only consider the case where $\gamma = 1$, $\alpha = \beta = z = 0$, with $x$ and $y$ nonzero. Under these conditions, $\widetilde{v_1} = (-c,-c)$ and $\widetilde{v_2} = (ax, by)$ and therefore:
$$
| \mathcal{R} | = 
\begin{vmatrix}
\; -c & ax \; \\
\; -c & by \;
\end{vmatrix}
= c\; |ax - by|.
$$
Again by following the logic from Case 1 and Case 2, it suffices to show that $|\mathcal{R}| > |\bar{\mathcal{H}}|$ when $\widetilde{v_2}$ is at its shortest. Since $a < b$, the length of $\widetilde{v_2}$ is minimized when $x = 2$ and $y = -1$. In this case:

\begin{equation} \label{R3}
|\mathcal{R}| = c\; |2a + b| = 2ac + bc. 
\end{equation}

The computation of an upper bound for $|\bar{\mathcal{H}}|$ in Case 3 is similar to those in Cases 1 and 2. Namely, it is computed by identifying lattice points in $\mathcal{H}$ via multiples of $(-c,-c)$ (45-degree diagonal translations). The computations are also done as before using the well-understood region $\mathcal{H}$, however it is more economical now to subtract off the number of repeat $\mathcal{R}$-coset representatives from $|\mathcal{H}|$, rather than count the cosets directly as in Cases 1 and 2. Figure \ref{sp5} indicates that:

\begin{figure}
\includegraphics[height=350pt]{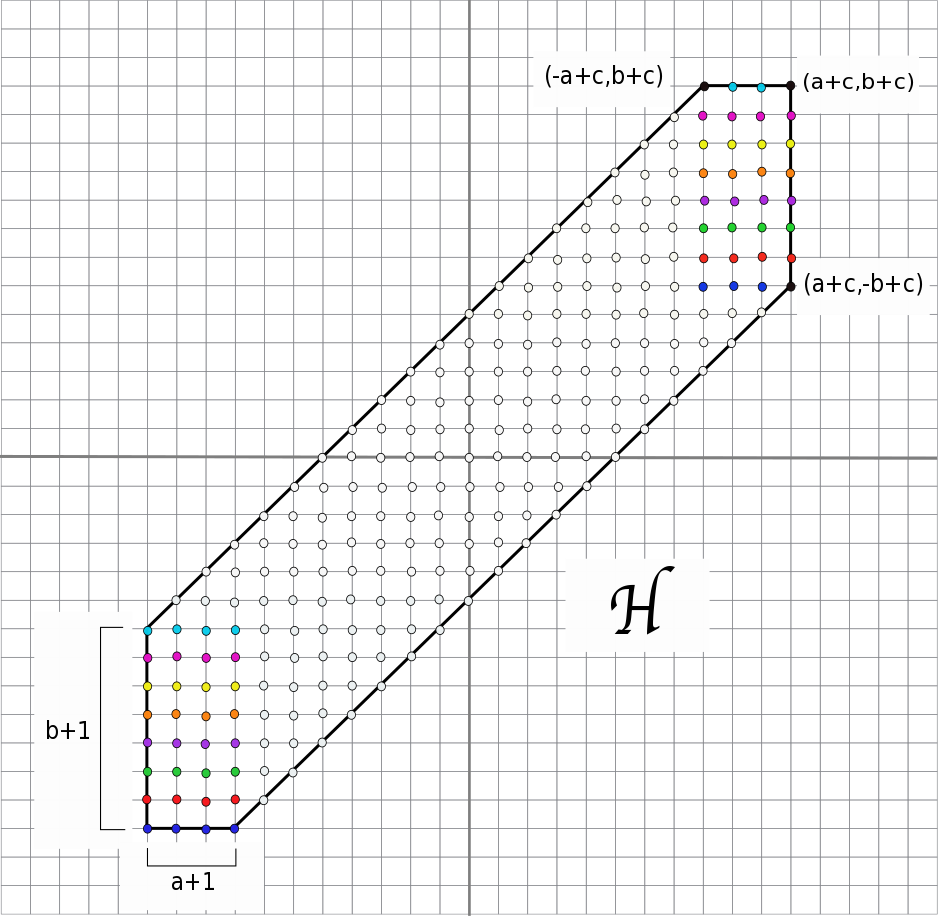}
\caption{$\mathcal{H}$ and its $\mathcal{R}$-cosets for $P(-3,-7, -19, 19, 55)$. Each white point represents a distinct $\mathcal{R}$-coset; colored points {\it lying along the same 45-degree diagonal} represent the same $\mathcal{R}$-coset.}
\label{sp5}
\end{figure} 

%\begin{figure}
%\includegraphics[height=400pt]{SinglePairC2.png}
%\caption{$\mathcal{H}$ and its $\mathcal{R}$-cosets for $P(-3,-7, 19,19,55)$, repeat representatives removed.}
%\label{sp6}
%\end{figure} 

$$
\begin{aligned}
	|\bar{\mathcal{H}}| & \leq |\mathcal{H}| - (a+1)(b+1)\\
			& = ab + ac + bc + a + b + c + 1 - (ab + a + b + 1)\\
			& = ac + bc + c.
\end{aligned}
$$
Comparing this result with Equation (\ref{R2}) gives the result:
$$
\begin{aligned}
	|\bar{\mathcal{H}}|	& \leq ac + bc + c\\
			& < 2ac + bc\\
			& = |\mathcal{R}|,
\end{aligned}
$$
since $c \geq a \geq 3$. Thus, 5-stranded odd pretzel knots of the form $P(-a,-b,-c,c,d)$, with $a,b, c, d \geq 3$, are not slice. The proof is complete.

 \qed

\bibliographystyle{amsalpha}
\bibliography{MutantRibbonPretzels}

\end{document}